\documentclass[11pt]{amsart}
\usepackage{amsmath,amssymb,mathrsfs,color}

\def\R{{\mathbb R}}
\def\d{{\rm d}}
\newtheorem{lemma}{Lemma}[section]
\newtheorem{theorem}[lemma]{Theorem}
\newtheorem{remark}[lemma]{Remark}

\newtheorem{coro}[lemma]{Corollary}

\newtheorem{example}[lemma]{Example}
\allowdisplaybreaks
\numberwithin{equation}{section}

\begin{document}

\title[Exponential ergocicity for McKean-Vlasov SDEs with switching]
{Existence, uniqueness and exponential ergodicity under Lyapunov conditions for McKean-Vlasov SDEs with Markovian switching}

\author{Zhenxin Liu}
\address{Z. Liu (Corresponding author): School of Mathematical Sciences,
Dalian University of Technology, Dalian 116024, P. R. China}
\email{zxliu@dlut.edu.cn}

\author{Jun Ma}
\address{J. Ma: School of Mathematical Sciences,
Dalian University of Technology, Dalian 116024, P. R. China}
\email{mathmajun@163.com}


\date{January 30, 2022}

\begin{abstract}
The paper is dedicated to studying the problem of existence and uniqueness of solutions as well as existence of and exponential convergence to invariant measures for McKean-Vlasov stochastic differential equations with Markovian switching.
Since the coefficients are only locally Lipschitz, we need to truncate them both in space and distribution variables simultaneously
to get the global existence of solutions under the Lyapunov condition. Furthermore, if the Lyapunov condition is strengthened, we establish the exponential convergence of solutions' distributions to the
unique invariant measure in Wasserstein quasi-distance and total variation distance, respectively. Finally, we give two applications to illustrate our theoretical results.
\end{abstract}

\maketitle

\section{Introduction}

Owing to the increasing demands on practical financial markets, ecological systems and social systems, much attention has been drawn to those processes which satisfy the McKean-Vlasov stochastic differential equation (MVSDE) with Markovian switching:
\begin{equation}\label{main}
dX_t= b(t,X_t,\mathcal{L}_{X_t},\alpha_t)dt+ \sigma(t,X_t,\mathcal{L}_{X_t},\alpha_t)dW_t
\end{equation}
and
$$
P(\alpha_{t+ \Delta t}=j | \alpha_t=i, (X_s,  \alpha_s ), s\leq t)
 = q_{ij}(X_t) \Delta t+ o(\Delta t)
$$
for $i\neq j$, where $\mathcal{L}_{X_t}$ denotes the law of $X_t$. A salient feature of such processes is the inclusion of the microcosmic site, the macrocosmic distribution of particles and the discrete event. For instance, the change rate of prices in a financial market may depend on the macrocosmic distribution, and may be very different for different time slots.

When the coefficients are Lipschitz and satisfy the linear growth condition, there are some related works on the stochastic system \eqref{main} with both McKean-Vlasov property (i.e. coefficients depending on the distribution) and Markovian switching property as follows. In \cite{ZSX}, Zhang et al proved the existence and uniqueness of Markov regime switching mean-field type stochastic control systems with state-independent switching in a finite state space. Nguyen et al \cite{NYH} showed that the limit of SDEs with mean-field interactions and Markovian switching is characterized as the stochastic McKean-Vlasov differential equation with Markovian switching in which the distribution term is actually the conditional distribution (given the history of the switching); the diffusion coefficient is assumed to be bounded. Nguyen et al \cite{NYN} obtained existence and uniqueness for conditional-distribution dependent stochastic control systems with state-independent switching in a finite state space.

If $b$ and $\sigma$ do not depend on Markovian switching $\alpha_t$, the equation \eqref{main} is called a McKean-Vlasov SDE or mean-field SDE. Such SDEs are used to study the interacting particle systems and mean-field games. It was first studied by Kac \cite{Kac} in the framework of the Boltzmann equation for the particle density in diluted monatomic gases, as well as in the stochastic toy model for the Vlasov kinetic equation for plasma. In \cite{Mckean}, McKean studied the propagation of chaos in physical systems of $N$-interacting particles related to the Boltzmann equation for the statistical mechanics of rarefied gases. In \cite{Sznitman1,Sznitman2}, Sznitman showed the propagation of chaos and the limit equation in a different framework. The limit equation can be described as an evolution equation known as the aforementioned MVSDE. The solution of a MVSDE is a ``nonlinear" Markov process, whose transition function may not only depends on the current state but also on the current distribution. Due to its importance and reality, the MVSDE is studied extensively. Larsy, Lions \cite{Larsy-Lions1,Larsy-Lions2,Larsy-Lions3} and Huang, Malhame and Caines \cite{HMC1,HMC2} independently introduced mean-field games in order to study large population deterministic and stochastic differential games. Veretennikov \cite{Veretennikov} obtained the existence and uniqueness of invariant measures and weak convergence to invariant measures for McKean-Vlasov SDEs with additive noise. Butkovsky \cite{Butkovsky} considered ergodic properties of nonlinear Markov chains and McKean-Vlasov equations with additive noise. Buckdahn et al \cite{BLPR} established the relationship between the functionals of the form $E f(t,X_t,\mathcal L_{X_t})$ and the associated second-order PDE, involving derivatives with respect to (w.r.t. in short) the law. In \cite{wang_2018}, Wang showed the well-posedness, existence and uniqueness of invariant measures under monotone conditions. Bogachev et al \cite{Bogachev-Rockner-Shaposhnikov} obtained convergence in variation of probability measure solutions to stationary measures of nonlinear Fokker-Planck equations. Mishura and Veretennikov \cite{MV} established weak and strong existence/uniqueness results for solutions of multi-dimensional MVSDEs under relaxed regularity conditions. Barbu and R\"ockner \cite{Barbu-Rockner} got the existence of weak solutions to MVSDEs using the superposition principle. Song \cite{Song} studied exponential ergodicity for MVSDEs with jumps. In \cite{RTW}, Ren et al proved the existence and uniqueness of solutions in infinite dimension under a Lyapunov condition (different from ours in the present paper).

Along another line, if $b$ and $\sigma$ do not depend on the distribution $\mathcal{L}_{X_t}$, then \eqref{main} reduces to the so-called switching diffusion system, also known as hybrid switching system, which has gained increasing popularity because of its ability to handle numerous real-world applications in which continuous and discrete dynamics coexist and interact. The behavior of a diffusion process in different environments may be very different. Thus, it can provide more opportunity for realistic models. For instance, the work \cite{BR} of Barone-Adesi and Whaley is one of the early efforts using switching processes for financial applications, and in \cite{YKI} an optimization problem leads to switching diffusion limits under suitable conditions. Yin and his cooperators have systematically studied switching diffusions, such as regularity, Feller property, recurrence, ergodicity and numerical approximation, see e.g. \cite{YZ,ZY_2009}. In \cite{YZ}, they established the existence and uniqueness of solutions under the Lipschitz condition and Lyapunov condition respectively, and ergodicity using cycles and induced Markov chains which is similar to the classical situation. Cloez and Hairer \cite{CM} proved the ergodicity with state-dependent switching in a finite state space, using the weak form of Harris' Theorem (Hairer et al \cite{HMS}). In \cite{Shao1} Shao obtained the ergodicity with state-independent switching in both finite and infinite state spaces, and in \cite{Shao2} he got the existence and uniqueness of strong solutions with switching in an infinite state space.

The main purpose of our paper is to investigate the existence and uniqueness of solutions as well as exponential ergodicity for the equation \eqref{main}, which we derive under Lyapunov type conditions in a unified way. Since the coefficients depend on the distribution of solutions which is a global property, the classical truncation in the space variable does not work in this situation. Following Ren et al \cite{RTW}, we need to truncate the equation in both space and distribution variables to overcome this difficulty. For the existence of and convergence to invariant measures, we do not appeal to the streamlined method of Hairer and Mattingly \cite{HM} which works for the convergence in Wasserstein distance as well as total variation distance and is now widely adopted such as in Bogachev et al \cite{Bogachev-Rockner-Shaposhnikov}, Wang \cite{Wang}. Instead, we use Lyapunov function itself to achieve the same goal, which we think is simple and interesting in its own right, and also consistent with our Lyapunov function method throughout the paper. Our method works also for the convergence in total variation distance, but the price we pay is that the convergence only works in Wasserstein quasi-distance instead of Wasserstein distance in \cite{HM}; see the comment following (H5) in Section 4 for details.

The rest of this paper is arranged as follows. In Section 2, we collect a number of preliminary results concerning switching, transition semigroup and optimal transportation cost. Section 3 presents existence and uniqueness under the Lyapunov condition. Section 4 establishes exponential convergence to invariant measures under the condition of integrable Lyapunov function, both in Wasserstein quasi-distance and weighted total variation distance. In Section 5, we provide two examples to illustrate our theoretical results.

\section{Preliminary}

Throughout the paper, let $(\Omega, \mathcal F, \{\mathcal F_t\}_{t\ge 0},P)$ be a filtered complete probability space. We assume that the filtration $\{\mathcal F_t\}$ satisfies the usual condition, i.e. it is right continuous and $\mathcal F_0$ contains all $P$-null sets. Let $W$ be an $n$-dimensional Brownian motion defined in $(\Omega, \mathcal F, \{\mathcal F_t\}_{t\ge 0}, P)$.  We denote by $B^\top$ the transpose of matrix $B\in \mathbb R^{n_1\times n_2}$ with $n_1,n_2\geq1$, $tr(B)$ the trace of $B$ and $|B|:=\sqrt {tr(B^\top B)}$ the norm of $B$. Suppose that $\alpha$ is a stochastic process with right-continuous sample paths, taking values in a finite set $\mathcal M = \{ 1,2,\cdots, m \}$, and having $x$-dependent generator $Q=(q_{i,j}): \mathbb R^d \to \mathbb R^{m \times m}$ such that for a suitable function $V(\cdot, \cdot)$,
\begin{align*}
Q(x) V(x, \cdot)(i)= \sum_{j\in \mathcal M} q_{ij}(x)V(x, j)
=\sum_{j\in \mathcal M,j\neq i} q_{ij}(x)(V(x, j)- V(x, i)), ~ x\in \mathbb R^d, ~i\in \mathcal M .
\end{align*}
We say $Q$ satisfies the $q$-property, if $q_{i,j}(\cdot)$ is Borel measurable, uniformly bounded, $q_{i,j}(x)\geq 0$ for $j\neq i$ and $q_{i,i}(x)= -\sum_{j\neq i}q_{i,j}(x)$ for all $i,j \in \mathcal M$ and $x\in \mathbb R^d$. Assume that $(X_t, \mathcal L_{X_t}, \alpha_t)_{t\ge 0}$ is a triplet such that $X_t$ is a continuous component taking values in $\mathbb R^d$, $\mathcal L_{X_t}$ denotes the distribution of $X_t$ taking values in $\mathcal P (\mathbb R^d)$ and $\alpha_t$ is a jump component taking values in $\mathcal M$, where $\mathcal P (\mathbb R^d)$ is the space of probability measures on $\mathbb R^d$. The process $(X_t, \alpha_t)$ can be described by the following MVSDE with switching:
\begin{align}
\left\{
\begin{aligned}\label{eqSDE}
dX_t&= b(t,X_t,\mathcal{L}_{X_t},\alpha_t)dt+ \sigma(t,X_t,\mathcal{L}_{X_t},\alpha_t)dW_t\\
X_0&= \xi, \alpha_0= \zeta,
\end{aligned}
\right.
\end{align}
and for $i \neq j$,
\begin{align}\label{eqPofS}
P(\alpha_{t+ \Delta t}=j | \alpha_{t}=i, (X_s, \alpha_s ), s\leq t)= q_{ij}(X_t) \Delta t+ o(\Delta t),
\end{align}
where $b:[0,\infty) \times \mathbb R^d \times \mathcal P (\mathbb R ^d) \times \mathcal{M} \to \mathbb R^d$, $\sigma:[0,\infty) \times \mathbb R^d \times \mathcal P (\mathbb R ^d) \times \mathcal{M} \to \mathbb R^{d\times n}$, and
$\xi$ is $\mathcal F_0$-measurable and satisfies some integrable condition to be specified below.
The MVSDE has a generator $ L$ given as follows. For each $i\in \mathcal M$ and any twice continuously differentiable function $V(\cdot, i)$,
$$
L V(x, i)= \frac{1}{2}tr(\sigma \sigma^\top \nabla ^2 V(x, i))+ b(t, x, \mu, i)\nabla V(x, i) + Q(x) V(x, \cdot)(i)
$$
where $x\in \mathbb R^d$, $\nabla ^2 V(\cdot, i)$ and $\nabla V(\cdot, i)$ denote the Hessian and gradient of $V(\cdot, i)$ respectively.

Note that the evolution of the discrete component $\alpha$ can be represented as a stochastic integral with respect to a Poisson random measure. Indeed, for any $x\in \mathbb R^d$ and $i,j \in \mathcal M$ with $i\neq j$, let $\Delta _{ij}(x)$ be consecutive (w.r.t. the lexicographic ordering on $\mathcal M \times \mathcal M$), left closed, right open interval of the real line, each having length $q_{ij}(x)$. Define a function $h: \mathbb R^d \times \mathcal M \times \mathbb R \to \mathbb R$ by
$$
h(x, i, z):=\left\{
\begin{aligned}
& j-i, && {z\in \Delta _{ij}(x)},\\
& 0, && {\rm else}.
\end{aligned}
\right.
$$
Then it is equivalent to
$$
d\alpha_t= \int_{\mathbb R} h(X_t, \alpha_{t-}, z)p(dt,dz),
$$
where $p(dt,dz)$ is a Poisson random measure with intensity $dt\times \tilde m (dz)$; here $\tilde m$ is the Lebesgue measure on $\mathbb R$. The poisson random measure $p(\cdot, \cdot)$ is independent of the Brownian motion $W$.

The generalized It\^o's formula reads
\begin{align*}
V(X_t,\alpha_t)- V(X_0,\alpha_0)= \int_0^t L V(X_s,\alpha_s)ds+ M_1(t)+ M_2(t),
\end{align*}
where
\begin{align*}
M_1(t)&= \int_0^t \langle \nabla V(X_s,\alpha_s), \sigma(s,X_s,\mathcal{L}_{X_s},\alpha_s)dW_s\rangle,\\
M_2(t)&= \int_0^t \int_{\mathbb R} V(X_s,\alpha_0+ h(X_s, \alpha_{s-}, z))- V(X_s,\alpha_s) \mu(ds,dz),
\end{align*}
and
\begin{align*}
\mu(ds,dz)= p(ds,dz)- ds\times \tilde m(dz).
\end{align*}

When strong existence and uniqueness of solutions holds for \eqref{eqSDE}-\eqref{eqPofS}, the solution $(X_t, \alpha_t)_{t\geq s}$ is a Markov process which is determined by solving the equation from $s$ with initial value $(X_s, \alpha_s)$. More precisely, denote by $\{X_{s,t}^{\alpha}(\xi)\}_{t\geq s}$ the solution of the equation from $s$ with initial value $X_{s,s}= \xi, \alpha_s= \alpha$, then the uniqueness implies
\begin{equation}\label{equniqueness}
X_{s,t}^{\alpha}(\xi)= X^{\alpha_r}_{r,t}(X_{s,r}^{\alpha}(\xi)), ~ 0\leq s \leq r \leq t.
\end{equation}
However, in general, the solution is not strong Markovian because we do not have $\mathcal L _{X_{\tau}}=\mathcal L _{X_t}$ on the set $\{\tau =t\}$ for a stopping time $\tau$ and $t\geq 0$. Moreover, the associated Markov operator $P_t$ given by
$$
P_t f(x, \alpha):=  E f(X_t ^{\alpha}(x), \alpha(t)), \quad \alpha \in \mathcal M,x\in \mathbb R^d,f\in \mathscr B_b(\mathbb R^d \times \mathcal M)
$$
is not a semigroup, where $\mathscr B_b(\mathbb R^d \times \mathcal M)$ denotes the space of bounded measurable functions on $\mathbb R^d \times \mathcal M$.

We will consider solutions of \eqref{eqSDE}--\eqref{eqPofS} with some integrable conditions, so let us introduce some basic notations as follows. Let $\mathcal P(\mathbb R^d \times \mathcal M)$ be the space of probability measures on $\mathbb R^d \times \mathcal M$, and $\rho:(\mathbb R^d \times \mathcal M) \times (\mathbb R^d \times \mathcal M) \to \mathbb R^+$ be a distance-like function satisfying $\rho((x,i),(y,j))=0$ if and only if $x=y,~i=j$.  Denote $\mathcal P_{\rho}:= \{\mu\in \mathcal P(\mathbb R^d \times \mathcal M): \int_{\mathbb R^d\times\mathcal M}\rho((x,i),(0,1))\mu(dx\times\{i\})<\infty\}$. If the weak uniqueness holds for \eqref{eqSDE}--\eqref{eqPofS} in $\mathcal P_{\rho}$, we may define a semigroup $P_{s,t}^*$ on $\mathcal P_{\rho}$ by letting $P_{s,t}^{*} \mu:= \mathcal L_{\{X_{s,t}, \alpha_{s,t}\}}$ for $\mathcal L_{\{X_s,\alpha_s\}}= \mu$. Then we have
$$
P_{s,t}^*= P_{r,t}^*P_{s,r}^* \quad \hbox{for }  0\leq s\leq r\leq t.
$$
Note that the semigroup $P_{s,t}^*$ is nonlinear, i.e.
\begin{align*}
P^*_{s,t}\mu \neq \int_{\mathbb R^d\times \mathcal M}(P_{s,t}^*\delta_{x,i})\mu(dx\times\{i\}), \quad 0\le s \le t.
\end{align*}
In the time homogeneous case, i.e. $b$ and $\sigma$ do not depend on $t$, we have $P^*_{s,t}= P^{*}_{t-s}$ for $0\leq s\leq t$. A measure $\mu \in \mathcal P_{\rho}$ is said to be invariant measure of $P_{t}^*$ if $P_{t}^* \mu= \mu$ for all $t\geq 0$, and the equation is said to be ergodic if there exists $\mu_{\mathcal I} \in \mathcal P_{\rho}$ such that $\lim_{t\to \infty}P_{t}^* \nu= \mu_{\mathcal I}$ weakly for any $\nu \in \mathcal P_\rho$. It is obvious that ergodicity implies uniqueness of invariant measures.

We now introduce the Wasserstein quasi-distance based on $\rho$. For any $\mu, \nu \in \mathcal P_{\rho}$, let
\begin{align*}
W_{\rho}(\mu, \nu):=\inf_{\pi \in \mathcal C(\mu, \nu)} \int_{(\mathbb R^d\times\mathcal M) \times (\mathbb R^d\times\mathcal M)}\rho((x,i), (y,j)) \pi(dx\times \{i\}, dy\times \{j\})=\inf E \rho(X, Y),
\end{align*}
where $\mathcal C(\mu, \nu)$ is the set of couplings between $\mu$ and $\nu$, and the second infimum is taken over all random variables $X, Y$ on $\mathbb R^d \times \mathcal M$ whose laws are $\mu, \nu$ respectively. In general, $W_{\rho}$ is not a distance because the triangle inequality may not hold. But it is complete in the sense that any $W_{\rho}$--Cauchy sequence in $\mathcal P_{\rho}$ is convergent, i.e. for any Cauchy sequence $\{\mu_n\} \subset \mathcal P_{\rho}$, there exists a measure $\mu \in \mathcal P_{\rho}$ such that $W_{\rho}(\mu_n, \mu) \to 0$ as $n\to \infty$.
When $\rho$ is a distance on $\mathbb R^d \times \mathcal M$, $W_{\rho}$ satisfies the triangle inequality and is hence a distance on $\mathcal P_{\rho}$.

We also use the usual Wasserstein distance $W_p$ on $\mathcal P_p(\R^d)$ with $p=1,2$ in what follows, i.e. $\mathcal P_p(\R^d):=\{\mu\in\mathcal P(\R^d): \int_{\R^d}|x|^p \mu(dx)<\infty \}$ and $W_p(\mu, \nu):=\inf_{\pi \in \mathcal C(\mu, \nu)} [\int_{\mathbb R^d\times\R^d} |x-y|^p \pi(dx, dy)]^{1/p}$ for $\mu, \nu\in \mathcal P_p(\R^d)$. This should not cause confusion with $W_\rho$ and $\mathcal P_\rho$ introduced above. As usual, we also denote $\mu(f):=\int_{\R^d} f(x)\mu(dx)$ in what follows for any function $f$ defined on $\R^d$ and $\mu\in\mathcal P(\R^d)$.

\section{Existence and uniqueness of solutions}

In this section, we consider the existence and uniqueness of the equation \eqref{eqSDE}--\eqref{eqPofS} under the Lyapunov function condition. Firstly, we consider the existence and uniqueness under Lipschitz and linear growth conditions.

\begin{theorem}\label{thMo1}
Suppose that $b:[0,\infty) \times \mathbb R^d \times \mathcal P (\mathbb R ^d) \times \mathcal{M} \to \mathbb R^d$, and $\sigma:[0,\infty) \times \mathbb R^d \times \mathcal P (\mathbb R^d) \times \mathcal{M} \to \mathbb R^{d\times n}$ are measurable, and satisfy the following conditions: for each $ t \in [0,\infty),\alpha \in \mathcal M,  x,y \in \mathbb R^d, \mu, \nu \in \mathcal P _2(\mathbb R ^d)$, there exist constants $ K,L>0 $ such that
\begin{align*}
|b(t, x, \mu, \alpha)- b(t, y, \nu, \alpha)| & \leq L\left(| x- y |  + W_2\left(\mu,\nu\right)\right),\\
| \sigma(t, x, \mu, \alpha)- \sigma(t, y, \nu, \alpha)| &\leq L\left(| x- y |  + W_2\left(\mu,\nu\right)\right),\\
| b(t, x, \mu, \alpha)| + | \sigma (t, x, \mu, \alpha)|
& \leq K \left( 1 +| x |  +\left(\mu \left(| \cdot | ^2\right)\right)^{\frac{1}{2}} \right).
\end{align*}
The generator $Q= (q_{i,j}): \mathbb R^d \to \mathbb R^{m\times m}$ is a bounded, continuous function and satisfy the $q$-property. Then for any $ T>0 $, $\alpha \in \mathcal M$ and $X_0 \in L^2(\Omega, \mathcal F_0,P)$,
\eqref{eqSDE} has a unique solution $(X_t, \alpha_t )$ with the given initial data in which the evolution of the jump process $\alpha_t$ is specified by \eqref{eqPofS} and $X_t$ satisfies
$$
 E \sup _{0 \leq t \leq T} | X_t |^2 < \infty.
$$
\end{theorem}

\begin{proof}
This result is known well. So for brevity we only outline main steps.

1. Uniqueness. Suppose $(X_t, \mathcal L_{X_t}, \alpha_t)$ and $(Y_t, \mathcal L_{Y_t}, \tilde\alpha_t)$ are solutions. If $\alpha_t=\tilde \alpha_t~ a.s.$, the uniqueness follows from It\^o's formula and Gronwall's inequality since the coefficients are Lipschitz. Otherwise, define $\tau :=\inf\{t\geq 0: \alpha_t= \tilde\alpha_t\}$. We can prove $\tau= \infty~ a.s.$ This proof is similar to the forthcoming Theorem \ref{thlL1} so we omit it.

2. Existence. Let $X^0_t= X_0,~ \mu^0_t= \mathcal L_{X_0}$. For any $n \geq 1$, let $X^n_t$ solve the SDE with Markovian switching
\begin{equation*}
\left\{
\begin{aligned}
\d X^n_t&= b(t,X^n_t,\mu^{n-1}_t,\alpha^{n}_t)dt+ \sigma(t,X^n_t,\mu^{n-1}_t,\alpha^{n}_t)dW_t\\
X^n_0 &= X_0, \alpha^{n}_0= \alpha \\
\end{aligned}
\right.
\end{equation*}
and for $ i\neq j$,
\begin{align*}
P(\alpha^{n} _{t+ \Delta t}=j | \alpha^{n}_t=i, (X^n_s, \alpha^{n}_s ), s\leq t)= q_{ij}(X^{n}_t) \Delta t+ o(\Delta t).
\end{align*}
As the coefficients are Lipschitz and satisfy the linear growth condition, we can prove that $E\sup_{0\leq t\leq T}|X^n_t|^2 < \infty$ and $\{ X^n_t\}$ is a Cauchy sequence, and hence has a limit $X_t$ in the space $C([0,T])$ as $n\to \infty$, which is a solution.
\end{proof}

Now we introduce some assumptions for the equation \eqref{eqSDE}--\eqref{eqPofS}.
\begin{enumerate}

\item [(H1)]For any $N\geq 1$, $\alpha \in \mathcal M$, there exists a constant $ C_N\geq 0 $ such that for any $ |x|, |y|\leq N$ and  ${\rm supp} \mu, {\rm supp} \nu \subset B(0,N)$ we have
\begin{align*}
&|b(t, x, \mu, \alpha) |+ |\sigma(t, x, \mu, \alpha)| \leq C_N,\\
&| b(t, x, \mu, \alpha)- b(t, y, \nu, \alpha) |+ |\sigma(t, x, \mu, \alpha)- \sigma(t, y, \nu, \alpha)| \leq  C_N(| x- y | + W_2(\mu,\nu)).
\end{align*}
Here $B(0,N)$ denotes the closed ball in $\R^d$ centered at the origin with radius $N$.

\item [(H2)](Lyapunov function)
There exists a function $ V:\mathbb R^d \times \mathcal M \to \mathbb R^{+}$ that is twice  continuously differentiable with respect to $x\in \mathbb R^d$ for each $i\in \mathcal M $ such that there exist constants $ \lambda_1,\lambda_2  \in \mathbb R $ satisfying for all $(t,x,\mu, i)\in [0,\infty)\times\mathbb R^d \times \mathcal P(\R^d)\times \mathcal M$
\begin{align*}
(LV)(t, x,\mu, i)  &\leq  \lambda_1 V(x, i) + \lambda_2 \int _{\mathbb R^d} \varphi(x)\mu(dx), \\
V_R&:= \inf_{|x|\geq R, i\in \mathcal M}V(x, i)\to \infty ~as~ R\to \infty,
\end{align*}
where function $\varphi: \mathbb R^d \to \mathbb R^+$ satisfying $\varphi(x)\leq V(x,i)$ for all $x\in \mathbb R^d, i\in \mathcal M$.

\item [(H3)](Continuity) For any $\alpha \in \mathcal M$ and bounded sequences $ \{{x_n, \mu_n}\} \in \mathbb R^d \times \mathcal P_V (\mathbb R^d) $ with $ x_n \to x $ and $\mu_n \to \mu$ weakly in $\mathcal P (\mathbb R^d)$ as $ n\to \infty $, we have
\begin{equation}\nonumber
\lim_{n\to \infty} \sup_{t\in [0,T]}{ | b(t, x_n, \mu _n, \alpha)- b(t, x, \mu, \alpha) |+ |(\sigma(t, x_n, \mu _n, \alpha)- \sigma(t, x, \mu, \alpha)|}= 0.
\end{equation}
where $ \mathcal P_V (\mathbb R^d):= \left\{ \mu \in \mathcal P(\mathbb R^d):\int_{\mathbb R^d} V(x,i)\mu(dx) <\infty, \forall i\in \mathcal M \right\}$.

\item [(H4)]There exist constants $ K,\epsilon >0 $ and increasing unbounded function $L: \mathbb N \to (0,\infty)$ such that for any $ \alpha\in \mathcal M,~N\geq 1,|x|\vee |y| \leq N $ and $  \mu,\nu \in \mathcal P (\mathbb R ^d ) $ satisfying
\begin{align*}
&| b(t, x, \mu, \alpha)- b(t, y, \nu, \alpha) |+ |\sigma(t, x, \mu, \alpha)- \sigma(t, y, \nu, \alpha)|\\
&\leq L_N(| x- y |+ W_{2,N}(\mu , \nu)+ Ke^{-\epsilon L_N}(1\wedge W_2(\mu,\nu))),
\end{align*}
where
\begin{align*}
W_{2,N}^2(\mu,\nu):= \inf_{\pi \in \mathcal C(\mu,\nu)} \int_{\mathbb R^d \times \mathbb R^d}|\phi_N(x) -\phi_N(y)|^2 \pi(dx,dy) ,
\phi_N (x):= \frac{Nx}{N\vee |x|}.
\end{align*}
\end{enumerate}

\begin{remark}\rm
If the function $L:\mathbb N\to (0,\infty)$ in (H4) is bounded, i.e. $b$ and $\sigma$ are globally Lipschitz, then $K$ should be $0$. This then reduces to the the case of Theorem \ref{thMo1}, so we assume that the function
$L$ is unbounded in (H4).
\end{remark}

\begin{theorem}\label{thlL1}
Assume (H1)--(H3). Then for any $ T> 0$, $ X_0 \in L^2(\Omega, \mathcal{F}_0, P)$ and $\alpha_0 \in \mathcal M$, \eqref{eqSDE}--\eqref{eqPofS} has a solution $(X_{\cdot},\alpha_{\cdot} )$ which satisfies
$$
E V( X_t, \alpha_t ) \leq e^{(\lambda_1 + \lambda_2) t}E V( X_0, \alpha_0 ), \quad \hbox{for } t\ge 0.
$$
Moreover, if (H4) holds, then the solution is unique.
\end{theorem}

\begin{proof}

(i) Existence.

1. In order to construct a solution using Theorem \ref{thMo1}, we take a sequence of truncations of  $b$ and $\sigma$ as follows. For any $ n\geq 1 ,t\in [0, T],  x\in \mathbb R ^d, \mu \in \mathcal P(\mathbb R ^d), \alpha \in \mathcal M$, define
\begin{align*}
b^n(t, x, \mu , \alpha)&:= b(t, \phi_n (x), \mu \circ \phi^{-1}_n,\alpha),\\
\sigma^n(t, x, \mu, \alpha )&:= \sigma(t, \phi_n (x), \mu \circ \phi^{-1}_n, \alpha).
\end{align*}
For each $n\geq 1$, $ b^n $ and $ \sigma^n $ are Lipschitz and satisfy the linear growth condition. Therefore, by Theorem \ref{thMo1}, the equation
\begin{equation}\label{eqlocalappsde}
\left\{
\begin{aligned}
\d X^n_t&= b^n(t,X^n_t,\mathcal{L}_{X^n_t},\alpha^{n}_t)dt+ \sigma^n(t,X^n_t,\mathcal{L}_{X^n_t},\alpha^{n}_t)dW_t\\
X^n_0 &= X_0, \alpha^{n}_0= \alpha_0 \\
\end{aligned}
\right.
\end{equation}
and for $ i\neq j$,
\begin{align}\label{eqlocalswitching}
P(\alpha^{n} _{t+ \Delta t}=j | \alpha^{n}_t=i, (X^n_s, \alpha^{n}_s ), s\leq t)= q_{ij}(X^{n}_t) \Delta t+ o(\Delta t)
\end{align}
has a unique solution $(X^n_t, \alpha^n_t)$.
Define $ \tau ^n:= \inf\{ t\geq 0: |X^n_t|\geq n \} $. By the definition of $ \phi_n $, we have
$$
\phi_n(X ^n _t)= \frac{X ^n _t \cdot n}{|X ^n _t| \vee n}= X ^n _{t\wedge \tau^n}.
$$
Moreover, for any measurable set $ A\subset \mathbb R^d $, we obtain
\begin{align*}
&{(\mathcal{L}_{X^n_t})\circ \phi_n^{-1}}(A)= P(X^n_t \in \phi_n^{-1}(A))\\
=& P(\phi_n(X^n_t)\in A)=\mathcal{L}_{\phi_n (X^n_t)} (A)= \mathcal{L}_{X^n_{t\wedge \tau^n}} (A).
\end{align*}
So the equation \eqref{eqlocalappsde}--\eqref{eqlocalswitching} becomes
\begin{equation}
\left\{
\begin{aligned}
\d X^n_t&= b(t, X^n_{t\wedge \tau^n}, \mathcal{L}_{X^n_{t\wedge \tau^n}}, \alpha^n_t)dt+ \sigma(t, X^n_{t\wedge \tau^n}, \mathcal{L}_{X^n_{t\wedge \tau^n}}, \alpha^n_t)dW_t\\
X^n_0&= X_0, \alpha^n_0= \alpha\\
\end{aligned}
\right.
\end{equation}
and for $ i\neq j$,
\begin{align}
P(\alpha^{n}_{t+ \Delta t}=j | \alpha^{n}_t=i, (X^n_s, \alpha^{n}_s ), s\leq t)= q_{ij}(X^{n}_t) \Delta t+ o(\Delta t).
\end{align}

2. Applying It\^o's formula to $ V( X ^n _t , \alpha^n_t) $, we have
\begin{align*}
&V( X ^n _t , \alpha^n_t)- V( X^n_0, \alpha^n_0)\\
=& \int_{0}^{t} L^nV( X ^n _s , \alpha^n_s)ds
 + \int_{0}^{t}\nabla V \sigma^n(s, X^n_{s}, \mathcal{L}_{X^n_{s}}, \alpha^n_s)dW_s\\
& +\int_{0}^{t}\int_{\mathbb R}V( X^n_s, \alpha_0+ h( X^n_s, \alpha^n_s, z)) - V( X^n_s, \alpha^n_s)\mu(ds,dz),
\end{align*}
where $L^n$ represents the generator corresponding to the $n$-th equation for $(X^n,\alpha^n)$. Thus, taking expectation on both sides, we get
\begin{align*}
&EV( X ^n _{t \wedge \tau^n} , \alpha^n_{t \wedge \tau^n})- EV( X^n_0, \alpha^n_0)\\
=& E\int_0 ^{t \wedge \tau^n} L^nV( X ^n _s , \alpha^n_s)ds
 = E\int_0 ^t \mathbb\mathcal 1_{[0,\tau^n]}(s) L^nV( X ^n _s , \alpha^n_s)ds\\
=& E\int_0 ^t \mathbb\mathcal 1_{[0,\tau^n]}(s) \nabla V( X ^n _{s\wedge \tau^n} , \alpha^n_{s\wedge \tau^n})b^n(s, X ^n _{s\wedge \tau^n} ,\mathcal L_{X ^n _s}, \alpha^n_{s\wedge \tau^n})ds\\
&+  \frac{1}{2}E\int_0 ^t \mathbb\mathcal 1_{[0,\tau^n]}(s) \nabla^2 V( X ^n _{s\wedge \tau^n} , \alpha^n_{s\wedge \tau^n})A^n(s, X ^n _{s\wedge \tau^n} ,\mathcal L_{X ^n _s}, \alpha^n_{s\wedge \tau^n})ds\\
&+ E\int_0 ^t \mathbb\mathcal 1_{[0,\tau^n]}(s)Q(X ^n _{s\wedge \tau^n})V( X ^n _{s\wedge \tau^n},\cdot )(\alpha^n_{s\wedge \tau^n})ds\\
=& E\int_0 ^t \mathbb\mathcal 1_{[0,\tau^n]}(s) \nabla V( X ^n _{s\wedge \tau^n} , \alpha^n_{s\wedge \tau^n})b(s, X ^n _{s\wedge \tau^n} ,\mathcal L_{X ^n _{s\wedge \tau^n}}, \alpha^n_{s\wedge \tau^n})ds\\
&+  \frac{1}{2}E\int_0 ^t \mathbb\mathcal 1_{[0,\tau^n]}(s) \nabla^2 V( X ^n _{s\wedge \tau^n} , \alpha^n_{s\wedge \tau^n})A(s, X ^n _{s\wedge \tau^n} ,\mathcal L_{X ^n _{s\wedge \tau^n}}, \alpha^n_{s\wedge \tau^n})ds\\
&+ E\int_0 ^t \mathbb\mathcal 1_{[0,\tau^n]}(s)Q(X ^n _{s\wedge \tau^n})V( X ^n _{s\wedge \tau^n},\cdot )(\alpha^n_{s\wedge \tau^n})ds\\
=& E\int_0 ^t \mathbb\mathcal 1_{[0,\tau^n]}(s)LV( X ^n _{s\wedge \tau^n} , \alpha^n_{s\wedge \tau^n})ds\\
\leq & E\int_0 ^t \mathbb\mathcal 1_{[0,\tau^n]}(s) [\lambda_1 V( X ^n _{s\wedge \tau^n}, \alpha^n_{s\wedge \tau^n})+ \lambda_2 E\varphi(X ^n _{s\wedge \tau^n})]ds\\
\leq & \int_0 ^t \mathbb\mathcal 1_{[0,\tau^n]}(s) (\lambda_1+ \lambda_2)EV( X ^n _{s\wedge \tau^n}, \alpha^n_{s\wedge \tau^n})ds,
\end{align*}
where
\begin{align*}
A^n(s, x,\mu,\alpha)&:=\sigma^n(s, x,\mu,\alpha)\sigma^n(s, x,\mu,\alpha)^\top,\\
A(s, x,\mu,\alpha)&:=\sigma(s, x,\mu,\alpha)\sigma(s, x,\mu,\alpha)^\top.
\end{align*}
Applying Gronwall's inequality, we get
\begin{align*}
EV( X ^n _{t\wedge \tau^n}, \alpha^n_{t\wedge \tau^n}) &\leq e^{(\lambda_1 + \lambda_2)t}EV( X ^n _0, \alpha^n_0)\\
& \leq e^{(\lambda_1 + \lambda_2)T}EV( X ^n _0, \alpha^n_0)=:\delta.
\end{align*}
Denote $ \tau_N^n:= \inf\{t\geq 0:|X^n_t|\geq N\}, ~ n\geq N\geq 1$ and let $t= T \wedge \tau_N^n $, we have
\begin{align*}
&E V( X^n _{T\wedge\tau_N^n} , \alpha^n_{T\wedge\tau_N^n})\leq e^{(\lambda_1 + \lambda_2)T}EV( X ^n _0, \alpha^n_0)=:\delta.
\end{align*}
Consequently, $\tau_N^n $ satisfies
$$
P(\tau_N^n < T)\leq \frac{\delta}{V(N,\alpha^n_ {\tau_N^n})}.
$$

3.
Let $l\ge 1$ to be determined. By (H1) and BDG's inequality, there exists a constant $ C(N,l)>0 $ such that for any $ n\geq N $ we have
\begin{align*}
& E(\sup_{t\in[s,(s+\epsilon)\wedge T]}| X^n_{t\wedge \tau_N^n}- X^n_{s\wedge \tau_N^n} |^{2l})\\
=& E\bigg(\sup_{t\in[s,(s+\epsilon)\wedge T]}\bigg|\int_{s\wedge \tau_N^n}^{t\wedge \tau_N^n} b^n(r,X^n_r,\mathcal{L}_{X^n_r},\alpha^n_r)dr
+ \int_{s\wedge \tau_N^n}^{t\wedge \tau_N^n} \sigma ^n(r,X^n_r,\mathcal{L}_{X^n_r},\alpha^n_r)dW_r \bigg|^{2l}\bigg)\\
\leq & C(l) E\bigg[\sup_{t\in[s,(s+\epsilon)\wedge T]} C_N \epsilon^{2l}+\sup_{t\in[s,(s+\epsilon)\wedge T]}\bigg(\int_{s\wedge \tau_N^n}^{t\wedge \tau_N^n}|\sigma^n(r,X^n_r,\mathcal{L}_{X^n_r},\alpha^n_r)|^2 dr\bigg)^l\bigg]\\
\leq & C(N,l) \epsilon^l.
\end{align*}
Let $ k=[\frac{T}{\epsilon}]+1$ where $[a]$ denotes the integer part of $a\in \mathbb R$. Then we obtain
\begin{align*}
&E(\sup_{s,t\in[0,T],|t-s|\leq \epsilon}| X^n_{t\wedge \tau_N^n}-X^n_{s\wedge \tau_N^n} |^{2l})\\
\leq & C(l)\sum_{j=1}^{k} E(\sup_{t\in[(j-1)\epsilon,j\epsilon \wedge T]}| X^n_{t\wedge \tau_N^n}-X^n_{(j-1)\epsilon\wedge \tau_N^n} |^{2l})\\
\leq & C(N,l)(T+\epsilon)\epsilon^{l-1}.
\end{align*}
Therefore, by H\"older's inequality we have
\begin{align*}
&E(\sup_{s,t\in[0,T],|t-s|\leq \epsilon}| X^n_{t\wedge \tau_N^n}- X^n_{s\wedge \tau_N^n} | )
\leq (C(N,l)(T+\epsilon))^{\frac{1}{2l}}\epsilon^{\frac{1}{2}-\frac{1}{2l}}.
\end{align*}
Taking $ l=2 $, we get
\begin{equation}
E(\sup_{s,t\in[0,T],|t-s|\leq \epsilon}| X^n_{t\wedge \tau_N^n}- X^n_{s\wedge \tau_N^n} | )
\leq (C(N)(T+\epsilon))^{\frac{1}{4}}\epsilon^{\frac{1}{4}}.
\end{equation}
When $ n=N, \tau_N^n=\tau_n^n=\tau^n $. By Arzela-Ascoli type theorem for measures, the sequence $\{\mu^n:= \mathcal{L}_{X^n_{t\wedge \tau^n}}\} $ is tight in $ \mathcal P (C[0, T])$. Therefore, by the Prokhorov theorem, there exists a subsequence, still denoted $ \{\mu^n\}$, such that $ \mu^n \to \mu $ weakly in $ \mathcal{P} (C[0, T]) $ as $ n\to \infty $.

4.
Define $ \tau_N^{n,m}:=\tau_N^{n}\wedge \tau_N^{m} $. Then for any $m \geq n \geq N$
\begin{align}\label{eqofphiX}
\phi_N (X^{j}_{t\wedge \tau_N^{n,m}})= X^{j}_{t\wedge \tau_N^{n,m}},~ j\in\{n,m\}
\end{align}
and
\begin{align}\label{eqofmuphi}
\lim_{n\to \infty} \sup_{m\geq n} \mu^{n}\circ \phi_{m}^{-1}= \mu \hbox{~ weakly ~in ~}\mathcal P (\mathbb R^d).
\end{align}
By Cauchy-Schwarz inequality and BDG's inequality, we arrive
\begin{align*}
&E( \sup_{0\leq s\leq t} | X^{n}_{s\wedge \tau_N^{n,m}}- X^{m}_{s\wedge \tau_N^{n,m}} |^2 )\\
\leq & 2E\bigg( \sup_{0\leq s\leq t}\bigg|\int_0^{s\wedge \tau_N^{n,m}} b^n(r,X^n_r,\mathcal{L}_{X^n_r},\alpha^n_r)- b^m(r,X^m_r,\mathcal{L}_{X^m_r},\alpha^m_r) dr \bigg|^2 \bigg)\\
& +
2E\bigg( \sup_{0\leq s\leq t} \bigg|\int_0^{s\wedge \tau_N^{n,m}} \sigma^n(r,X^n_r,\mathcal{L}_{X^n_r},\alpha^n_r)- \sigma^m(r,X^m_r,\mathcal{L}_{X^m_r},\alpha^m_r) dW_r\bigg|^2\bigg)\\
\leq & 2TE \int_0^{t\wedge \tau_N^{n,m}} |b^n(r,X^n_r,\mathcal{L}_{X^n_r},\alpha^n_r)- b^m(r,X^m_r,\mathcal{L}_{X^m_r},\alpha^m_r)|^2 dr \\
& +
CE \int_0^{t\wedge \tau_N^{n,m}} |\sigma^n(r,X^n_r,\mathcal{L}_{X^n_r},\alpha^n_r)- \sigma^m(r,X^m_r,\mathcal{L}_{X^m_r},\alpha^m_r)|^2 dr \\
=& 2TE \int_0^{t\wedge \tau_N^{n,m}} |b(r,X^n_{r\wedge \tau^n},\mu^n_r,\alpha^n_r)- b(r,X^m_{r\wedge \tau^m},\mu^m_r,\alpha^m_r)|^2 dr \\
& +
CE \int_0^{t\wedge \tau_N^{n,m}} |\sigma(r,X^n_{r\wedge \tau^n},\mu^n_r,\alpha^n_r)- \sigma(r,X^m_{r\wedge \tau^m},\mu^m_r,\alpha^m_r)|^2 dr \\
\leq &
4TE \int_0^{t\wedge \tau_N^{n,m}} |b(r,X^n_{r\wedge \tau^n},\mu^n_r,\alpha^n_r)- b(r,X^m_{r\wedge \tau^m},\mu^m_r,\alpha^n_r)|^2 dr \\
& +
4TE \int_0^{t\wedge \tau_N^{n,m}} |b(r,X^m_{r\wedge \tau^m},\mu^m_r,\alpha^n_r)- b(r,X^m_{r\wedge \tau^m},\mu^m_r,\alpha^m_r)|^2 dr \\
& +
CE \int_0^{t\wedge \tau_N^{n,m}} |\sigma(r,X^n_{r\wedge \tau^n},\mu^n_r,\alpha^n_r)- \sigma(r,X^m_{r\wedge \tau^m},\mu^m_r,\alpha^n_r)|^2 dr \\
& +
CE \int_0^{t\wedge \tau_N^{n,m}} |\sigma(r,X^m_{r\wedge \tau^m},\mu^m_r,\alpha^n_r)- \sigma(r,X^m_{r\wedge \tau^m},\mu^m_r,\alpha^m_r)|^2 dr.
\end{align*}
By \eqref{eqofphiX}, \eqref{eqofmuphi} and (H1), there exists a family of constants $ \{ \epsilon_{n,m}: m\ge n\geq 1 \} $ with $ \epsilon_{n,m}\to 0 $ as $ n\to \infty $ such that
\begin{align*}
& | b(t,X^{n}_{t\wedge \tau_N^{n,m}},\mu^{n}_t,\alpha^{n}_t)- b(t,X^{m}_{t\wedge \tau_N^{n,m}}, \mu^{m}_t, \alpha^{n}_t) |\\
\leq &
| b(t, X^{n}_{t\wedge \tau_N^{n,m}}, \mu^{n}_t\circ \phi_{n}^{-1}, \alpha^{n}_t)- b(t, X^{m}_{t\wedge \tau_N^{m}}, \mu^{n}_t\circ \phi_{n}^{-1}, \alpha^n_t) |\\
& + | b(t,X^{m}_{t\wedge \tau_N^{n,m}}, \mu^{n}_t\circ \phi_{n}^{-1}, \alpha^{n}_t)- b(t, X^{m}_{t\wedge \tau_N^{n,m}}, \mu^{m}_t\circ \phi_{m}^{-1}, \alpha^{n}_t) |\\
\leq &
C_N \cdot | X^{n}_{t\wedge \tau_N^{n,m}}- X^{m}_{t\wedge \tau_N^{n,m}} | + C_N  \cdot \epsilon_{n,m}
\end{align*}
when $n\geq N$.
Next, we treat the term with different switching. Partition the interval $[0,T]$ by $\epsilon _{n,m}$ (for short $\epsilon$). We obtain
\begin{align}\label{ieqofdifferentswitching}
\nonumber &E\int_0^{t\wedge \tau_N^{n,m}} |b(r,X^m_{r\wedge \tau^m},\mu^m_r,\alpha^n_r)- b(r,X^m_{r\wedge \tau^m},\mu^m_r,\alpha^m_r)|^2 dr \\\nonumber
\leq &
E\sum_{k=0}^{[T/\epsilon]} \int_{k\epsilon }^{(k+1)\epsilon }\mathcal\mathbb 1_{[0,\tau_N^{n,m}]}(r)|b(r,X^m_{r\wedge \tau^m},\mu^m_r,\alpha^n_r)- b(r,X^m_{r\wedge \tau^m},\mu^m_r,\alpha^m_r)|^2 dr\\
\leq &
3\sum_{k=0}^{[T/\epsilon]} E\int_{k\epsilon }^{(k+1)\epsilon }\mathcal\mathbb 1_{[0,\tau_N^{n,m}]}(r)|b(r,X^m_{r\wedge \tau^m},\mu^m_r,\alpha^n_r)- b(r,X^m_{k\epsilon \wedge \tau^m},\mu^m_r,\alpha^n_r)|^2 dr\\\nonumber
& +
3\sum_{k=0}^{[T/\epsilon]} E\int_{k\epsilon }^{(k+1)\epsilon }\mathcal\mathbb 1_{[0,\tau_N^{n,m}]}(r)|b(r,X^m_{k\epsilon \wedge \tau^m},\mu^m_r,\alpha^n_r)- b(r,X^m_{k\epsilon \wedge \tau^m},\mu^m_r,\alpha^m_r)|^2 dr\\\nonumber
& +
3\sum_{k=0}^{[T/\epsilon]} E\int_{k\epsilon }^{(k+1)\epsilon }\mathcal\mathbb 1_{[0,\tau_N^{n,m}]}(r)|b(r,X^m_{k\epsilon \wedge \tau^m},\mu^m_r,\alpha^m_r)- b(r,X^m_{r\wedge \tau^m},\mu^m_r,\alpha^m_r)|^2 dr.
\end{align}
For the first term of the right-hand side of \eqref{ieqofdifferentswitching}, by the local Lipschitz continuity of coefficient $b$ we have
\begin{align*}
&\sum_{k=0}^{[T/\epsilon ]} E\int_{k\epsilon }^{(k+1)\epsilon }\mathcal\mathbb 1_{[0,\tau_N^{n,m}]}(r)|b(r,X^m_{r\wedge \tau^m},\mu^m_r,\alpha^n_r)- b(r,X^m_{k\epsilon \wedge \tau^m},\mu^m_r,\alpha^n_r)|^2 dr\\
\leq &
\sum_{k=0}^{[T/\epsilon ]} E\int_{k\epsilon }^{(k+1)\epsilon }\mathcal\mathbb 1_{[0,\tau_N^{n,m}]}(r)C_N^2|X^m_{r\wedge \tau^m}- X^m_{k\epsilon \wedge \tau^m}|^2 dr\\
\leq &
\sum_{k=0}^{[T/\epsilon ]}C(N)\epsilon^2
\leq  C(N,T)\epsilon.
\end{align*}
In the same way, for the last term of \eqref{ieqofdifferentswitching} we get
\begin{align*}
E\sum_{k=0}^{[T/\epsilon ]} \int_{k\epsilon}^{(k+1)\epsilon }\mathcal\mathbb 1_{[0,\tau_N^{n,m}]}(r)|b(r,X^m_{k\epsilon \wedge \tau^m},\mu^m_r,\alpha^m_r)- b(r,X^m_{r\wedge \tau^m},\mu^m_r,\alpha^m_r)|^2 dr
\leq C(N,T)\epsilon.
\end{align*}
As for the second term of the right-hand side of \eqref{ieqofdifferentswitching}, we have for each $k=0,1,...,[T/\epsilon]$
\begin{align}\label{ieqofdifferentswitching1}
\nonumber &E\int_{k\epsilon}^{(k+1)\epsilon }\mathcal\mathbb 1_{[0,\tau_N^{n,m}]}(r)|b(r,X^m_{k\epsilon \wedge \tau^m},\mu^m_r,\alpha^n_r)- b(r,X^m_{k\epsilon \wedge \tau^m},\mu^m_r,\alpha^m_r)|^2 dr\\
\leq &
2E\int_{k\epsilon }^{(k+1)\epsilon }\mathcal\mathbb 1_{[0,\tau_N^{n,m}]}(r)|b(r,X^m_{k\epsilon \wedge \tau^m},\mu^m_r,\alpha^n_r)- b(r,X^m_{k\epsilon \wedge \tau^m},\mu^m_r,\alpha^m_{k\epsilon }|^2 dr\\\nonumber
& +
2E\int_{k\epsilon }^{(k+1)\epsilon}\mathcal\mathbb 1_{[0,\tau_N^{n,m}]}(r)|b(r,X^m_{k\epsilon \wedge \tau^m},\mu^m_r,\alpha^m_{k\epsilon })- b(r,X^m_{k\epsilon \wedge \tau^m},\mu^m_r,\alpha^m_r)|^2 dr.
\end{align}
For the second term of the right-hand side of $\eqref{ieqofdifferentswitching1}$, we get
\begin{align*}
&E\int_{k\epsilon }^{(k+1)\epsilon }\mathcal\mathbb 1_{[0,\tau_N^{n,m}]}(r)|b(r,X^m_{k\epsilon \wedge \tau^m},\mu^m_r,\alpha^m_{k\epsilon })- b(r,X^m_{k\epsilon \wedge \tau^m},\mu^m_r,\alpha^m_r)|^2 dr\\
= &
E\sum_{j\neq i,i\in \mathcal M}\int_{k\epsilon }^{(k+1)\epsilon }\mathcal\mathbb 1_{[0,\tau_N^{n,m}]}(r)|b(r,X^m_{k\epsilon \wedge \tau^m},\mu^m_r,i)- b(r,X^m_{k\epsilon \wedge \tau^m},\mu^m_r,j)|^2\mathcal\mathbb 1_{\{\alpha^m_r=j\}}\mathcal\mathbb 1_{\{\alpha^m_{k\epsilon }=i\}} dr\\
\leq &
E\sum_{j\neq i,i\in \mathcal M}\int_{k\epsilon }^{(k+1)\epsilon }4C_N^2E(\mathcal\mathbb 1_{\{\alpha^m_r=j\}}| X^m_{k\epsilon }, \alpha^m_{k\epsilon }=i) \mathcal\mathbb 1_{\{\alpha^m_{k\epsilon }=i\}}dr\\
\leq &
4C_N^2 E\sum_{i\in \mathcal M}\int_{k\epsilon }^{(k+1)\epsilon } \mathcal\mathbb 1_{\{\alpha^m_{k\epsilon }=i\}}\cdot \bigg(\sum_{j\neq i}q_{ij}(X^m_{k\epsilon })(r- {k\epsilon })+ o(r- {k\epsilon })\bigg) dr\\
\leq &
C(N,M)\epsilon ^2,
\end{align*}
where $M$ denotes the bound of $Q$. To treat the first term of the right-hand side of \eqref{ieqofdifferentswitching1}, we use the technique of basic coupling of Markov processes. Denote by $\tilde Q(x_1,x_2):=(\tilde q_{(k,l),(j,i)}(x_1,x_2))$ the basic coupling of $Q(x_1)$ and $Q(x_2)$, which satisfies
\begin{align*}
\tilde Q(x_1,x_2)f(k,l)
&=
\sum_{(j,i)\in \mathcal M \times \mathcal M}\tilde q_{(k,l),(j,i)}(x_1,x_2)(f(j,i)-f(k,l))\\
&=
\sum_{j\in \mathcal M}(q_{kj}(x_1)- q_{lj}(x_2))^+(f(j,l)-f(k,l))\\
& \quad +
\sum_{j\in \mathcal M }(q_{lj}(x_2)-q_{kj}(x_1))^+(f(k,j)-f(k,l))\\
& \quad +
\sum_{j\in \mathcal M}(q_{kj}(x_1)\wedge q_{lj}(x_2))(f(j,j)-f(k,l))
\end{align*}
for any function $f: \mathcal M \times \mathcal M \to \mathbb R $. Consequently, let $(\alpha^n_t,\alpha^m_t)$ be a stochastic process on a finite state space $\mathcal M \times \mathcal M$ with generator $\tilde Q(x_1,x_2)$. Then for any $i_1,i_2,j\in \mathcal M$ with $j\neq i_2$, $r\in [k\epsilon, k\epsilon+ \epsilon)$ we have
\begin{align*}
&E(\mathcal\mathbb 1_{\{ \alpha^n_r=j\}} | \alpha^n_{k\epsilon }=i_1, \alpha^m_{k\epsilon }=i_2,X^n_{k\epsilon }=x_1, X^m_{k\epsilon }=x_2 )\\
=&
\sum_{l\in \mathcal M}E(\mathcal\mathbb 1_{\{\alpha^n_r=j\}} \mathcal\mathbb 1_{\{\alpha^m_r=l\}} | \alpha^n_{k\epsilon}=i_1, \alpha^m_{k\epsilon }=i_2,X^n_{k\epsilon }=x_1, X^m_{k\epsilon }=x_2 )\\
= &
\sum_{l\in \mathcal M}\tilde q_{(i_1,i_2),(j,l)}(x_1,x_2)(r-k\epsilon )+ o(r-k\epsilon )
\leq
m \tilde M\epsilon,
\end{align*}
where $\tilde M$ denotes the bound of $\tilde Q$.
Thus, for the first term of \eqref{ieqofdifferentswitching1} we obtain
\begin{align*}
&E\int_{k\epsilon }^{(k+1)\epsilon }\mathcal\mathbb 1_{[0,\tau_N^{n,m}]}(r)|b(r,X^m_{k\epsilon \wedge \tau^m},\mu^m_r,\alpha^n_r)- b(r,X^m_{k\epsilon \wedge \tau^m},\mu^m_r,\alpha^m_{k\epsilon }|^2 dr\\
= &
E\sum_{j\neq i,i\in \mathcal M}\int_{k\epsilon }^{(k+1)\epsilon }\mathcal\mathbb 1_{[0,\tau_N^{n,m}]}(r)|b(r,X^m_{k\epsilon \wedge \tau^m},\mu^m_r,j)- b(r,X^m_{k\epsilon \wedge \tau^m},\mu^m_r,i)|^2\mathcal\mathbb 1_{\{\alpha^n_r=j\}} \mathcal\mathbb 1_{\{\alpha^m_{k\epsilon }=i\}} dr\\
\leq &
E\sum_{j\neq i,i,i_1\in \mathcal M}\int_{k\epsilon }^{(k+1)\epsilon } 4C_N^2 \mathcal\mathbb 1_{\{\alpha^m_{k\epsilon }=i,\alpha^n_{k\epsilon }=i_1\}} E(\mathcal\mathbb 1_{\{\alpha^n_r=j\}} | \alpha^n_{k\epsilon }=i_1, \alpha^m_{k\epsilon }=i,X^n_{k\epsilon }=x_2, X^m_{k\epsilon}=x)dr\\
\leq &
C(N,\tilde M)\epsilon^2.
\end{align*}
So we obtain
\begin{align*}
E\int_0^{t\wedge \tau_N^{n,m}} |b(r,X^m_{r\wedge \tau^m},\mu_m,\alpha^n_r)- b(r,X^m_{r\wedge \tau^m},\mu_m,\alpha^m_r)|^2 dr
\leq C(N,T,M,\tilde M)\epsilon_{n,m}.
\end{align*}
Similarly, we have for any $m\geq n\geq N$
\begin{align*}
&|\sigma(t,X^{n}_{t\wedge \tau_N^{n,m}},\mu^{n}_t,\alpha^{n}_t)- \sigma(t,X^{m}_{t\wedge \tau_N^{n,m}},\mu^{m}_t,\alpha^{n}_t)|\\
\leq & C_N | X^{n}_{t\wedge \tau_N^{n,m}}-X^{m}_{t\wedge \tau_N^{n,m}} | + C_N \epsilon_{n,m}
\end{align*}
and
\begin{align*}
E\int_0^{t\wedge \tau_N^{n,m}} |\sigma(r,X^m_{r\wedge \tau^m},\mu^m_r,\alpha^n_r)- \sigma(r,X^m_{r\wedge \tau^m},\mu^m_r,\alpha^m_r)|^2 dr
\leq
C(N,T,M,\tilde M)\epsilon_{n,m}.
\end{align*}
Now we arrive
\begin{align*}
&E( \sup_{0\leq s\leq t} | X^{n}_{s\wedge \tau_N^{n,m}}- X^{m}_{s\wedge \tau_N^{n,m}} |^2 )\\
\leq &
C(N,T) \int_{0}^{t}E| X^{n}_{s\wedge \tau_N^{n,m}}- X^{m}_{s\wedge \tau_N^{n,m}} |^2 ds+ C(N,T,M,\tilde M) \epsilon_{n,m}, ~ m\geq n\geq N.
\end{align*}
By Gronwall's inequality, we get
\begin{equation*}
E(\sup_{0\leq s\leq t} | X^{n}_{s\wedge \tau_N^{k,l}}- X^{m}_{s\wedge \tau_N^{k,l}} |^2)
\leq e^{ C(N,T)t}C(N,T,M,\tilde M)\epsilon_{n,m},~ m\geq n\geq N.
\end{equation*}
Then it follows that
\begin{equation}\label{ieqmcs2}
\lim_{n\to \infty} \sup_{m\geq n}E(\sup_{0\leq t\leq T}| X^{n}_{t\wedge \tau_N^{n,m}}- X^{m}_{t\wedge \tau_N^{n,m}} |^2)
\leq \lim_{n\to \infty} \sup_{m\geq n} e^{C(N,T)T}C(N,T,M,\tilde M) \epsilon_{n,m}= 0.
\end{equation}
Therefore, for any $ \epsilon> 0,m\geq n\geq N $, we obtain
\begin{align*}
& P(\sup_{0\leq t\leq T}| X^{n}_{t\wedge \tau^{n}}- X^{m}_{t \wedge \tau^{m}} | > \epsilon)\\
\leq &
P(\tau_N^{n}< T)+ P(\tau_N^{m}< T)+ P(\sup_{0\leq t\leq T}| X^{n}_{t\wedge \tau_N^{n,m}}- X^{m}_{t\wedge \tau_N^{n,m}} | > \epsilon)\\
\leq &
\frac{\delta}{V(N, \alpha^n( \tau_N^n))}+\frac{\delta}{V(N, \alpha^m( \tau_N^m))}+ P(\sup_{0\leq t\leq T}| X^{n}_{t\wedge \tau_N^{n,m}}- X^{m}_{t\wedge \tau_N^{n,m}} | > \epsilon).
\end{align*}
Combining this with \eqref{ieqmcs2}, for any $N\geq 1, \epsilon > 0$ we have
$$
\lim_{n\to \infty} \sup_{m\geq n}P(\sup_{0\leq t\leq T}|X^{n}_{t}- X^{m}_{t} | > \epsilon)
\leq
\frac{\delta}{V(N, \alpha^n( \tau_N^n))}+\frac{\delta}{V(N, \alpha^m( \tau_N^m))}.
$$
Letting $ N\to \infty $, we get that $ X^{n}_{\cdot } $ converges to a process $ X_{\cdot} $ in probability uniformly in $ [0, T] $. Therefore, there exists a subsequence, still denoted $\{ X_n\}$, such that $P$-$a.s.$
$$
\lim_{n\to \infty} \sup_{0\leq t\leq T}|X^ {n}_{t}- X_t |= 0.
$$
Especially, $\mathcal L_{X^ {n}_{t}} \to \mathcal{L}_{X_t} $ weakly in $ \mathcal P (C[0, T]) $. By the uniqueness of limit, we have $\mathcal{L}_{X_t}= \mu_t,~ t\in [0, T] $.
Therefore, combining this with (H1) and (H3), we let $n\to \infty$ in \eqref{eqlocalappsde}--\eqref{eqlocalswitching} to conclude that $X_{\cdot}$ satisfies
$$
X_t= X_0+ \int_{0}^{t }b(s,X_s,\mathcal{L}_{X_s},\alpha_s)ds+ \int_{0}^{t}\sigma(s,X_s,\mathcal{L}_{X_s},\alpha_s)dW_s,
$$
and for $i\neq j$,
$$
P(\alpha_{t+ \Delta t}=j | \alpha_{t}=i, (X_s, \alpha_s ), s\leq t)= q_{ij}(X_t) \Delta t+ o(\Delta t).
$$

5.
By It\^o's formula and (H2), we have
\begin{align*}
EV( X_t, \alpha_t)- EV( X_0, \alpha_0)= E\int_0 ^t LV( X_s, \alpha_s)ds
\leq  (\lambda_1+\lambda_2)E\int_0 ^tV( X_s, \alpha_s)ds.
\end{align*}
The estimate mentioned in the theorem now follows from Gronwall's inequality.

(ii). Uniqueness.

Assume that $(X_t, \alpha_t)$ and $(Y_t, \tilde\alpha_t )$ are two solutions with the same initial value.

1. If $ \alpha_t = \tilde\alpha_t, ~t\geq 0, a.s.$

We first prove the pathwise uniqueness up to a time $ t_0\in [0,T] $. Define the stopping time
$$
\tau_n:= \tau _n^X \wedge \tau _n^Y= \inf\{t\geq 0: | X(t) |\vee | Y(t) |\geq n\}, \quad n\geq 1.
$$
Then by (H4) and BDG's inequality we have
\begin{align*}
& E| X(t\wedge \tau_n)- Y(t\wedge \tau_n) |^2\\
& \leq 2T E \int_{0}^{t\wedge \tau_n}| b(s, X_s, \mathcal{L}_{X_s}, \alpha_s)- b(s, Y_s, \mathcal{L}_{Y_s}, \alpha_s) |^2 ds\\
& \quad + C E \int_{0}^{t\wedge \tau_n} |\sigma(s,X_s,\mathcal{L}_{X_s},\alpha_s)-\sigma(s,Y_s,\mathcal{L}_{Y_s},\alpha_s) |^2 ds\\
& \leq ( 2L_n T+ C \cdot L_n )E \int_{0}^{t\wedge \tau_n} |X_s- Y_s|^2+ W_{2,n}(\mathcal L _{X_s}, \mathcal L _{Y_s})^2+ K e^{-L_n \epsilon}ds\\
& \leq 2( L_n T+ C \cdot L_n )E \int_{0}^{t} |X_{s\wedge \tau_n}- Y_{s\wedge \tau_n}|^2+ K e^{-L_n \epsilon}ds.
\end{align*}
Applying Gronwall's inequality, we get
\begin{align*}
E| X(t\wedge \tau_n)- Y(t\wedge \tau_n) |^2 &\leq e^{2( L_n T+ C \cdot L_n )t}2( L_n T+ C \cdot L_n )T K e^{-L_n \epsilon}\\
&=2( L_n T+ L_n )T K e^{-L_n(\epsilon- 2( T+ 1 )t)}.
\end{align*}
Therefore, letting $n\to \infty $ and using Fatou's lemma, we get the uniqueness up to the time $t_0:={\frac{\epsilon}{2(T+ 1)} \wedge T}$.

If $t_0=T$, the proof is finished. Otherwise, because of $ X_{t_0}= Y_{t_0} $, we can use the same method to prove that the uniqueness holds up to the time $2t_0\wedge T$. Repeating this procedure, we can prove the uniqueness up to the time $ T $.

2. If $\alpha_t$ and $\tilde\alpha_t$ are not equal almost surely. Define $ \tau := \inf \left\{t\geq 0: \alpha_t \neq \tilde\alpha_t\right\}$, we want to prove $\tau= \infty ~ a.s.$ Obviously, this is equivalent to $\tau \wedge N= N$ for any $N>0$. Let $\eta := \tau \wedge N$ and $E:=\left\{ \omega: \eta(\omega)< N\right\}$.

Claim: $ P(E)=0$.

Indeed, if $ P(E)>0$, then for $a.s.~ \omega \in E$ we have
\begin{align*}
X_s(\omega)= Y_s(\omega), \alpha_s(\omega)= \tilde \alpha_s(\omega), ~\forall~ s\leq \tau(\omega) < N.
\end{align*}
Let $\eta_{\alpha}= \inf \left\{ s> \eta: \alpha_s \neq \alpha_{\eta}\right\}$, $\eta_{\tilde \alpha}= \inf \left\{ s> \eta: \tilde\alpha_s \neq \tilde\alpha_{\eta}\right\}$. By the definition of $\eta_{ \alpha}, \eta_{\tilde \alpha}$ and $\eta$, we have $\eta_{\alpha} \geq \eta, \eta_{\tilde \alpha}\geq \eta$ and there exists $\delta>0$ such that
\begin{align*}
\inf_{x\in \mathbb R^d, i\in \mathcal M}P(\eta_{\alpha} >\eta+ \delta: \alpha_{\eta}=i, X_{\eta}=x )\geq 1-\frac{1}{4}P(B),\\
\inf_{y\in \mathbb R^d, i\in \mathcal M}P(\eta_{\tilde\alpha} >\eta+ \delta: \tilde\alpha_{\eta}=i, Y_{\eta}=y )\geq 1-\frac{1}{4}P(B).
\end{align*}
Therefore, we get
\begin{align*}
&P(\eta_{\alpha} >\eta+ \delta)\\
&= \int _{\mathbb R^d \times \mathcal M}P(\eta_{\alpha} >\eta+ \delta: \alpha_{\eta}=i, X_{\eta}=x )P((X_{\eta}, \alpha_{\eta})\in (dx,di))\\
&\geq 1- \frac{1}{4}P(B).
\end{align*}
In the same way, we have
\begin{align*}
P(\eta_{\tilde\alpha} >\eta+ \delta) \geq 1- \frac{1}{4}P(B).
\end{align*}
Thus, we arrive
\begin{align*}
P(\left\{\eta_{\tilde\alpha} >\eta+ \delta\right\}\cap B) \geq P(\eta_{\tilde\alpha} >\eta+ \delta)- P(B^c) \geq \frac{3}{4}P(B)>0.
\end{align*}
Moreover, we obtain
\begin{align*}
P(\left\{\eta_{\alpha} >\eta+ \delta\right\} \cap \left\{\eta_{\tilde\alpha} >\eta+ \delta\right\} \cap B) \geq 1- \frac{1}{4}P(B)-(1- \frac{3}{4})P(B)= \frac{1}{2}P(B)>0.
\end{align*}
Define $\tilde\eta := \min\left\{ \eta_{\alpha}, \eta_{\tilde\alpha} \right\}$ and
$ \tilde\tau:= \tilde\eta \mathbb\mathcal 1_{\tau \leq M} + \zeta \mathbb\mathcal 1_{\tau > M}$. Then
 we get
\begin{align*}
P(\{ \tilde\tau> \tau \} \cap B)\geq P(\left\{\tilde\eta >\eta+ \delta\right\}\cap B)>0,
\end{align*}
and if $ \tau\leq M$, we have $\tilde\tau= \tilde\eta\geq \tau= \gamma = \tau \wedge M$. Therefore, there exists a subset $A$ of $B$ such that $\tau< \tilde\tau$ and $\alpha_t= \tilde\alpha_t$ for any $t\leq \tilde\tau$. This contradicts the definition of $\tau$. The proof is complete.
\end{proof}

\begin{remark}\rm
(i) When the Lyapunov function $V$ is independent of switching, we can choose $\varphi= V$ in (H2).

(ii) When $\mathcal M =\{1\}$, i.e. there is no switching in the equation \eqref{eqSDE}--\eqref{eqPofS}.
Comparing with Ren at al \cite{RTW}, it seems that our Lyapunov function condition is simpler; note also that their Lyapunov function cannot grow faster than $|x|^2$, while our condition has no this kind of restriction.
By taking $|x|^2$ as the Lyapunov function in our Theorem \ref{thlL1}, our result reduces to that of Hu \cite[Theorem 2.1]{Hu}.
\end{remark}

\section{Invariant measures and exponential convergence}

In this section, we investigate long time behaviors of solutions to \eqref{eqSDE}--\eqref{eqPofS}, i.e. the existence and uniqueness of invariant measures and exponential convergence to them. We divide this section into two parts: $\mathcal M= \{1\}$ and $\mathcal M = \{1,2,...,m\}$.

\subsection{The MVSDE case}
We first consider the special case $\mathcal M= \{1\}$, i.e. MVSDEs.
\begin{enumerate}\label{THMIPM}
\item [(H5)](Integrable Lyapunov condition)
There exists a function $  \tilde V :\mathbb R^d \to \mathbb R^{+}$ which is twice continuously differentiable w.r.t. $ x\in \mathbb R^d$ and satisfies $\tilde V(x)=0$ if and only if $x= 0$, such that there is a constant $ \gamma> 0 $ satisfying for each $\pi \in \mathcal C(\mu, \nu)$,
\begin{align}\label{ieqILC}
\int_{\mathbb R^d \times \mathbb R^d}\tilde L\tilde V( x-y )\pi(dx, dy)
\leq
-\gamma \int_{\mathbb R^d \times \mathbb R^d}\tilde V(x-y)\pi(dx, dy),
\end{align}
where $\tilde L \tilde V$ is defined by
\begin{align*}
\tilde L \tilde V(x-y):=&(b(t, x, \mu, 1)-b(t,y, \nu, 1))\nabla\tilde V(x-y)\\
& + \frac{1}{2}tr(\nabla^2\tilde V(x-y) A(t,x, y, \mu, \nu,1)),
\end{align*}
with $A(t,x, y, \mu, \nu,1)=(\sigma(t, x, \mu, 1)-\sigma(t,y,\nu,1))(\sigma(t, x, \mu, 1)-\sigma(t,y,\nu,1))^\top$.
\end{enumerate}

The function $\tilde V$ induces naturally a Wasserstein quasi-distance which is given by
$$
W_{\tilde V}(\mu, \nu):= \inf_{\pi \in \mathcal C(\mu, \nu)} \int_{\mathbb R^d \times \mathbb R^d}\tilde V(x-y)\pi(dx, dy) \quad \hbox{for } \mu, \nu \in \mathcal P_{\tilde V},
$$
where $\mathcal P_{\tilde V}:=\{ \mu\in \mathcal P(\mathbb R^d) : \mu(\tilde V)< \infty \}$.
In general, $W_{\tilde V}$ is not a distance because the triangle inequality may not hold. But it is complete in the sense that any $W_{\tilde V}$--Cauchy sequence in $\mathcal P _{\tilde V}$ is convergent.
When $d(x,y):=\tilde V(x-y)$ is a distance on $\mathbb R^d$, $W_{\tilde V}$ satisfies the triangle inequality and is hence a distance on $\mathcal P_{\tilde V}$. In what follows, we will study the exponential
ergodicity under this distance, which is simple and different from that of Hairer and Mattingly \cite{HM}. They used a Lyapunov function to construct a family of distances on both the state space and the probability measure
space to conclude the exponential ergodicity in total variation distance, which is now extensively adopted.

We have the following result on invariant measures and exponential convergence for MVSDEs.

\begin{theorem}\label{thinvariantmeasureofDDSDE}
Assume that (H1)--(H5) hold and $\mathcal M=\{1\}$.

(i) For any initial measures $ \mu_0 ,\nu_0 \in \mathcal P _{\tilde V}$, we have for $t\ge 0$
$$
W_{\tilde V}\left( P_t^*\mu_0, P_t^*\nu_0\right) \leq e^{-\gamma t} W_{\tilde V}\left( \mu_0, \nu_0 \right).
$$

(ii) If the coefficients $b$ and $\sigma$ are independent of $t$ and there exists $\nu_0 \in \mathcal P _{\tilde V}$ such that
\begin{align}\label{bounded}
\sup_{t\geq0} W_{\tilde V}\left( P_t^*\nu_0, \nu_0 \right)<\infty,
\end{align}
then there exists a unique invariant measure $\mu_{\mathcal I} \in \mathcal P_{\tilde V}$ such that
$$
W_{\tilde V}\left( P_t^*\mu_0, \mu_{\mathcal I} \right) \leq e^{-\gamma t} W_{\tilde V}\left( \mu_0, \mu_{\mathcal I}  \right) \quad \hbox{for } t \geq 0,~\mu_0 \in \mathcal P_{\tilde V}.
$$
\end{theorem}

\begin{proof}
(i). For any initial measures $\mu_0, \nu_0 \in \mathcal P_{\tilde V}(\mathbb R^d)$. Let $X_t$ and $Y_t$ be two solutions such that $\mathcal L_{X_0}= \mu_0, \mathcal L_{Y_0}= \nu_0$, and
$$
W_{\tilde V}(\mu_0, \nu_0)= E\tilde V(X_0- Y_0).
$$
Denote $\mu_t= \mathcal L_{X_t}, \nu_t= \mathcal L_{Y_t}$. By It\^o's formula and (H5), we have
\begin{align*}
E\tilde V(X_t- Y_t)
& =
E\tilde V(X_0- Y_0)+ E\int_0^t \tilde L\tilde V(X_s- Y_s) ds\\
& \leq
E\tilde V(X_0- Y_0)- \gamma E\int_0^t \tilde V(X_s- Y_s) ds.
\end{align*}
Applying Gronwall's inequality, we get
$$
E\tilde V(X_t- Y_t)
\leq
e^{-\gamma t}E\tilde V(X_0- Y_0).
$$
Thus,
$$
W_{\tilde V}(\mu_t, \nu_t)
\leq
e^{-\gamma t} W_{\tilde V}(\mu_0, \nu_0).
$$

(ii). We first prove that $\{P_t ^*\nu_0\}$ is a $W_{\tilde V}$--Cauchy sequence. Indeed, from (i) we know that
\begin{align*}
W_{\tilde V}\left( P_t^*\nu_0, P_{t+s}^*\nu_0\right)
\leq
e^{-\gamma t} W_{\tilde V}\left(\nu_0, P_{s}^*\nu_0\right).
\end{align*}
Thus, by \eqref{bounded} we obtain
\begin{align*}
\lim_{t\to \infty}\sup_{s\geq 0}W_{\tilde V}\left( P_t^*\nu_0, P_{t+s}^*\nu_0\right)=0.
\end{align*}

Since $\mathcal P_{\tilde V}$ is complete w.r.t. $W_{\tilde V}$, there exists a measure $\mu_{\mathcal I} \in \mathcal P_{\tilde V}$ such that
$$
\lim_{t\to \infty}W_{\tilde V}\left( P_t^*\nu_0, \mu_{\mathcal I} \right)=0.
$$
Consequently, by Lemma 4.2 in Villani \cite{Villani}, we have for any $t\ge 0$
\begin{align*}
W_{\tilde V}\left( P_t^*\mu_{\mathcal I} , \mu_{\mathcal I}\right)
\leq
\varliminf_{s \to \infty}W_{\tilde V}\left( P_t^*P_s^*\nu_0, \mu_{\mathcal I} \right) = 0.
\end{align*}
That is, $\mu_{\mathcal I} $ is an invariant measure. Therefore, by (i) for any $\mu_0\in \mathcal P_{\tilde V}$, we have
$$
W_{\tilde V}\left( P_t^*\mu_0, \mu_{\mathcal I} \right) \leq e^{-\gamma t} W_{\tilde V}\left( \mu_0, \mu_{\mathcal I}  \right).
$$
The proof is complete.
\end{proof}

\begin{remark}\rm
(i) The condition \eqref{bounded} means that there is a ``bounded orbit" in $\mathcal P _{\tilde V}$, which is necessary and natural because the system cannot have an invariant measure if any orbit is unbounded. Note by Theorem \ref{thinvariantmeasureofDDSDE}--(i) that existence of one ``bounded orbit" implies the boundedness of all the orbits in $\mathcal P _{\tilde V}$.

(ii) If inequality (\ref{ieqILC}) in (H5) is replaced by
\begin{align*}
\tilde L\tilde V( x-y )
\leq
-\gamma_1\tilde V(x-y)+ \gamma_2 W_{\tilde V}(\mu, \nu)
\end{align*}
with $\gamma_1>\gamma_2>0$, it is immediate to see that the results of Theorem \ref{thinvariantmeasureofDDSDE} are still valid.

(iii) By taking $\tilde V(\cdot)= |\cdot|^2$ in (H5), our result Theorem \ref{thinvariantmeasureofDDSDE} reduces to that of Hu \cite[Theorem 4.1]{Hu}, which in turn is a type of generalization of Wang \cite[Theorem 3.1]{wang_2018}.
Wang \cite{Wang} considered the exponential ergodicity under the Lyapunov and monotone conditions; note that the diffusion coefficient in \cite{Wang} requires to be non-degenerate and independent of the distribution, while our results do not need these assumptions.
\end{remark}

As a direct consequence of Theorem \ref{thinvariantmeasureofDDSDE}, we have
\begin{coro}\label{thinvariantmeasureunderTV}
Under the conditions of Theorem \ref{thinvariantmeasureofDDSDE}, for any measures $ \mu_0 ,\nu_0 \in \mathcal P _{\tilde V}$ we have
$$
\|P_t^*\mu_0- P_t^*\nu_0\|_{{\rm Var}, \tilde V} \to 0,~as~ t\to \infty.
$$
And there exists a unique invariant measure $\mu_{\mathcal I} \in \mathcal P_{\tilde V}$ such that for any measure $ \mu_0 \in \mathcal P _{\tilde V}$,
$$
\|P_t^*\mu_0- \mu_{\mathcal I}\|_{{\rm Var}, \tilde V} \to 0,~as~ t\to \infty.
$$
Here,
$$
\|\mu- \nu\|_{{\rm Var}, \tilde V}:= \sup_{|f|\leq \tilde V, f\in C_b} |\mu(f)- \nu(f)| \quad \hbox{for }\mu,\nu\in \mathcal P _{\tilde V}.
$$
\end{coro}

\begin{proof}
According to the Kantorovich duality (see e.g. \cite{Villani}), we have
\begin{align*}
\|\mu- \nu\|_{{\rm Var}, \tilde V} \leq \sup_{\phi, \psi \in C_b, \atop \phi- \psi\leq \tilde V}(\nu(\phi)- \mu(\psi))= W_{\tilde V}(\mu, \nu)
\end{align*}
for any $\mu,\nu \in \mathcal P_{\tilde V}$. Combining this with Theorem \ref{thinvariantmeasureofDDSDE}, the result immediately follows.
\end{proof}

\subsection{The case of MVSDEs with switching}
Next, we consider MVSDEs with Markovian switching, i.e. $\mathcal M =\{1,2,...,m\}$. For each fixed environment $i\in \mathcal M$, the corresponding diffusion process $X_t^{(i)}$ is defined by
$$
dX^{(i)}_t = b(t, X^{(i)}_t, \mathcal L_{X_t}, i )dt+ \sigma(t, X^{(i)}_t, \mathcal L_{X_t}, i )dW_t.
$$
Note that it should be $\mathcal L_{X_t}$ instead of $\mathcal L_{X^{(i)}_t}$ in above equation. Let $Y_t^{(i)}$ be defined the same as $X_t^{(i)}$
and denote by $\tilde L^{(i)}$ the infinitesimal generator of $X_t^{(i)}-Y_t^{(i)}$, i.e. for any twice continuously differentiable function $f:\mathbb R^d \to \mathbb R^+ $
\begin{align*}
\tilde L^{(i)} f(x-y):=&(b(t, x, \mu, i)-b(t,y, \nu, i))\nabla f(x-y) + \frac{1}{2}tr(\nabla^2 f(x-y) A(t,x, y, \mu, \nu,i))
\end{align*}
with $A(t,x, y, \mu, \nu,i)=(\sigma(t, x, \mu, i)-\sigma(t,y,\nu,i))(\sigma(t, x, \mu, i)-\sigma(t,y,\nu,i))^\top$.
\begin{enumerate}\label{THMIPM}
\item [(H6)](Integrable Lyapunov condition)
There exists a function $  \hat V:\mathbb R^d \to \mathbb R^{+}$, which is twice continuously differentiable with respect to $ x\in \mathbb R^d$, $ \hat V(x)=0$ iff $x=0$ and $\hat V(x-y)\leq K\max \{ \hat V(x), \hat V(y)\}$ for some constant $K$
and all $x,y\in\mathbb R^d$, such that there is a constant $ \theta> 0 $ satisfying  for each $\pi \in \mathcal C(\mu, \nu), i\in \mathcal M$,
\begin{align}\label{ieqILCS}
\int_{\mathbb R^d \times \mathbb R^d}\tilde L^{(i)}\hat V( x-y )\pi(dx, dy)
& \leq -\theta \int_{\mathbb R^d \times \mathbb R^d}\hat V(x-y)\pi(dx, dy).
\end{align}
\end{enumerate}

Let
\begin{align*}
&d((x,i), (y,j)):= \sqrt {\mathcal\mathbb 1_{i \neq j}+ \hat V(x- y)},~x,y \in \mathbb R^d, ~i,j \in \mathcal M,\\
&\mathcal P_d:=\left\{ \mu \in \mathcal P(\mathbb R^d \times \mathcal M): \int_{\mathbb R^d \times \mathcal M} d((x,i), (0,1))\mu(dx\times \{i\})< \infty \right\}.
\end{align*}
Define a Wasserstein quasi-distance on $\mathcal P_d$ by
$$
W_d(\mu, \nu):= \inf E d(X,Y )\quad \hbox {for }\mu, \nu \in \mathcal P_d,
$$
where the infimum is taken over all random variables $X, Y$ on $\mathbb R^d \times \mathcal M$ whose laws are $\mu, \nu$ respectively. It is complete in the space $\mathcal P_d$, i.e. any $W_d$--Cauchy sequence in $\mathcal P_d$ converges w.r.t. $W_d$.
Note that $W_d$ is a distance on $\mathcal P_d$ when $d$ is a distance on $\mathbb R^d \times \mathcal M$. In particular, when the mapping $(x,y)\mapsto \hat V(x-y)$ is a distance on $\mathbb R^d$, $d$ is a distance on $\mathbb R^d \times \mathcal M$
and $\hat V(x-y)\leq 2\max \{ \hat V(x), \hat V(y)\}$.


Denote
\begin{align*}
\mathcal P_{\hat V}:=\left\{ \mu \in \mathcal P(\mathbb R^d \times \mathcal M): \int_{\mathbb R^d \times \mathcal M}  \hat V(x) \mu(dx\times \{i\})< \infty \right\}.
\end{align*}

\begin{theorem}\label{thinvariantmeasure}
Assume that (H1)--(H4) and (H6) hold, $Q(x) \equiv Q$ and for any $\nu \in \mathcal P_{\hat V}$ we have $\sup_{t\ge 0}(P_t ^* \nu)(\hat V)< \infty$.
Then there exists a constant $\tilde\theta>0$ such that for any initial measures $ \mu_0 ,\nu_0 \in \mathcal P_{\hat V}$ we have
$$
W_d \left( P_t^*\mu_0, P_t^*\nu_0 \right) \leq Ce^{-\tilde\theta t}, \quad  t\ge 0
$$
for some constant $C=C(\mu_0,\nu_0)$. In particular, if the coefficients $b$ and $\sigma$ are independent of $t$, then there exists a unique invariant measure $\mu_{\mathcal I} \in \mathcal P_{\hat V}$ such that
for any $\mu_0 \in \mathcal P_{\hat V}$ we have
$$
W_d\left( P_t^*\mu_0, \mu_{\mathcal I} \right) \leq Ce^{-\tilde\theta t}, \quad t \geq 0
$$
with $C=C(\mu_0)$.
\end{theorem}

\begin{proof}

(i). Suppose that $(X_t, \alpha_t)$ and $(Y_t, \tilde\alpha_t)$ are solutions whose initial distributions are $\mu_0, \nu_0 $ respectively. Denote $\mu_t= \mathcal L_{(X_t,\alpha_t)}, \nu_t= \mathcal L_{(Y_t,\tilde\alpha_t)}$.

1. We first consider the special case $ \alpha_0= \tilde\alpha_0~ a.s.$ Then $\alpha_t= \tilde\alpha_t~ a.s.$ In this case, by (H6) and It\^o's formula we get
\begin{align*}
&E\hat V(X_t- Y_t)\\
= &
E\hat V(X_0- Y_0)+ E\int_0^t \tilde L^{(\alpha_s)}\hat V(X_s- Y_s) ds\\
\leq &
E\hat V(X_0- Y_0)- \theta E\int_0^t\hat V(X_s- Y_s) ds.
\end{align*}
Applying Gronwall's inequality, we obtain
\begin{align}\label{dissi}
E\hat V(X_t- Y_t)\leq  E\hat V(X_0- Y_0)e^{-\theta t}.
\end{align}
Thus, by Jensen's inequality we have
\begin{align*}
 W_d( \mu_t, \nu_t)
& \leq Ed((X_t, \alpha_t),(Y_t, \tilde\alpha_t))=
E \sqrt {\hat V(X_t- Y_t)}\\
& \leq
\sqrt {E\hat V(X_t- Y_t)}\leq
\sqrt {E\hat V(X_0- Y_0 )e^{-\theta t}}\\
& \le
e^{-\frac{\theta t}{2}}\sqrt {KE\hat V(X_0)+ KE\hat V(Y_0 )}.
\end{align*}
2. If $\alpha_0 = \tilde\alpha_0 ~a.s.$  does not hold, define $\tau := \inf \left\{ t\geq 0: \alpha_t= \tilde\alpha_t \right\}$. Recall that if $\alpha_t$ and $\tilde\alpha_t$ are two independent finite-state Markov chains with generator $Q$, then there exist constants $C_c, \theta_c>0$ such that
\begin{align*}
P(\tau>t) \leq C_c e^{-\theta_c t},~\forall t\geq 0.
\end{align*}
Thus, by H\"older's inequality, Jensen's inequality and \eqref{dissi} there exists a constant $C>0$ such that
\begin{align*}
&E d((X_t, \alpha_t), (Y_t, \tilde\alpha_t))\\
= &
E \left( \sqrt{\mathbb\mathcal 1_{\alpha_t \neq \tilde\alpha_t}+ \hat V(X_t- Y_t)}\cdot \mathbb\mathcal 1_{\tau> \frac{t}{2}}\right)
+ E \left(\sqrt { \hat V(X_t- Y_t)} \cdot \mathbb\mathcal 1_{\tau\leq \frac{t}{2}}\right)\\
\leq &
\sqrt {P(\tau> \frac{t}{2})} \cdot \sqrt{E(1+ \hat V(X_t- Y_t))}+ \sqrt { E(\hat V(X_t- Y_t) \mathbb\mathcal 1_{\tau\leq \frac{t}{2}})}\\
\leq &
\sqrt {C_c} e^{-\frac{\theta_c t}{4}} \sqrt{1+ K\max \{E\hat V(X_t),E\hat V( Y_t)\}}+ \sqrt { E(E(\hat V(X_t- Y_t)|\mathcal F_{\tau}) \mathbb\mathcal 1_{\tau\leq \frac{t}{2}})}\\
\leq &
\sqrt {C_c} e^{-\frac{\theta_c t}{4}} \sqrt{1+ K\max \{E\hat V(X_t),E\hat V( Y_t)\}}+ \sqrt {E\hat V(X_{\tau}- Y_{\tau}) e^{-\frac{\theta}{2}t} }\\
\leq &
C e^{-\tilde \theta t},
\end{align*}
where $\tilde \theta = \frac14(\theta \wedge \theta_c)$.

(ii). The proof is completely similar to that of Theorem \ref{thinvariantmeasureofDDSDE}--(ii), so we omit it.
\end{proof}

\begin{remark}\rm
We have the following comments on Theorem \ref{thinvariantmeasure}.

(i) If inequality (\ref{ieqILCS}) in (H6) is replaced by
\begin{align*}
\tilde L^{(i)}\hat V( x-y )
\leq
-\theta_1\hat V(x-y)+ \theta_2 W_{\hat V}(\mu, \nu)
\end{align*}
with $\theta_1>\theta_2>0$, the results are still valid.

(ii) The condition $\sup_{t\ge 0}(P_t ^* \nu)(\hat V)< \infty$ for any $\nu \in \mathcal P_{\hat V}$ means all the orbits in $\mathcal P_{\hat V}$ are bounded, which is natural and necessary to guarantee that the system is ergodic in $\mathcal P_{\hat V}$.

(iii) We assume $\mu_0\in \mathcal P_{\hat V}$ instead of $\mu_0\in \mathcal P_d$ since $\mu_0\in \mathcal P_d$ does not guarantee $\mu_0(\hat V)<\infty$.

(iv) In \cite{YZ}, Yin and Zhu showed the ergodicity for SDE with Markovian switching using the classical Khasminskii's method. But in the present paper, the solution of MVSDE with Markovian switching is not strong Markovian, so the classical Khasminskii's method does not apply; on the other hand, the corresponding Fokker-Planck equation is nonlinear, so the classcial Krylov-Bogolyubov argument for the existence of invariant measures is invalid, either.
\end{remark}

As a direct consequence of Theorem \ref{thinvariantmeasure}, we have the following corollary whose proof is omitted since it is similar to that of Corollary \ref{thinvariantmeasureunderTV}.
\begin{coro}\label{thinvariantmeasurestichingunderTV}
Under the conditions of Theorem \ref{thinvariantmeasure}, for any measures $ \mu_0 ,\nu_0 \in \mathcal P _{\hat V}$, we have
$$
\|P_t^*\mu_0- P_t^*\nu_0\|_{{\rm Var}, d} \to 0,~~as~ t\to \infty.
$$
And there exist a unique invariant measure $\mu_{\mathcal I} \in \mathcal P_{\hat V}$ such that for any measure $ \mu_0 \in \mathcal P _{\hat V}$,
$$
\|P_t^*\mu_0- \mu\|_{{\rm Var}, d} \to 0,~as~ t\to \infty.
$$
Here $\|\mu- \nu\|_{{\rm Var}, d}:= \sup_{|f|\leq d, f\in C_b} |\mu(f)- \nu(f)|$ for $\mu,\nu \in \mathcal P_d$.
\end{coro}

\section{Applications}

In this section, we provide two examples to illustrate our results.
\begin{example}
For each $x\in \mathbb R$, $\mu \in \mathcal P (\mathbb R)$ and $i\in \{1,2\}$, consider
\begin{align*}
b(x, \mu, 1)&= -x^3 -2\int_{\mathbb R}(x+ \beta y)\mu(dy),~ b(x, \mu, 2)= -2x ,\\
\sigma(x, \mu, 1)&= \int_{\mathbb R}(x+ \beta y)\mu(dy),~\sigma(x, \mu, 2)= x ,
\end{align*}
where $\beta \in \mathbb R$. Then the following results hold:
(i) the SDE
$$
dX_t = b( X_t, \mathcal L_{X_t}, \alpha_t )dt+ \sigma( X_t, \mathcal L_{X_t}, \alpha_t )dW_t
$$
has a unique solution for any $\beta \in \mathbb R$ and when $ E|X_0|^2 <\infty$ we have
$$
E|X_t|^2 \leq  e^{(-2+ 2\beta^2)t}E|X_0|^2, \quad \hbox{for } t\ge 0.
$$
(ii) If there is no switching and $\beta \in(-1,1)$, there exists a unique invariant measure to which the solutions' distributions are exponentially convergent under $W_2$ and $\|\cdot\|_{{\rm Var}, |\cdot|^2}$. Moreover, if the switching's generator is state-independent and $\beta \in(-1,1)$, there exists a unique invariant measure to which the solutions' distributions are exponentially convergent under $W_d$ and $\|\cdot\|_{{\rm Var},d}$, where $d((x,i)(y,j))= \sqrt {\mathcal\mathbb 1_{i\neq j}+ |x-y|^2}$ for $x,y\in \mathbb R, ~ i,j\in \{1,2\}$.
\end{example}

\begin{proof}
(i) It is immediate to see that (H1), (H3) and (H4) hold. We now check the assumptions (H2), (H5) and (H6). Let $V(x, i)= x^2$ for $x\in\R$ and $i=1,2$. By It\^o's formula and Cauchy-Schwarz inequality, we have
\begin{align*}
 LV(x, 1)
& =  \sigma^2( x, \mu, 1 )+ b( x, \mu, 1 )\cdot 2x \\
& \leq (\int_{\mathbb R}(x+ \beta y)\mu(dy))^2 - 4x \int_{\mathbb R}(x+ \beta y)\mu(dy)\\
& = -3x^2 -2\beta x\int_{\mathbb R} y \mu(dy) + \beta^2(\int_{\mathbb R} y\mu(dy))^2\\
& \leq -2 x^2 +2\beta^2 \int_{\mathbb R} x^2 \mu(dx),
\end{align*}
And in the same way we get
$$
 LV(x, 2)= -3x^2.
$$
Thus, in this example, $\varphi(x)= V(x)= x^2$ for $x\in\mathbb R$, $\lambda_1= -2,~\lambda_2= 2\beta^2$, i.e. (H2) holds. Therefore, by Theorem \ref{thlL1} there exists a unique solution $(X_t,\alpha_t)$ and we have
$$
E|X_t|^2 \leq  e^{(-2+ 2\beta^2)t}E|X_0|^2,\quad \hbox{for } t\ge 0.
$$
(ii)
By It\^o's formula and Cauchy-Schwarz inequality, we have
\begin{align*}
\int_{\mathbb R \times \mathbb R}  \tilde L^{(1)}V(x- y) \pi(dx, dy)
\leq & -2\int_{\mathbb R \times \mathbb R}  |x-y|^2 \pi(dx, dy)\\
\quad & + 2\beta^2|\int_{\mathbb R} x\mu(dx)- \int_{\mathbb R} y\nu(dy)|^2\\
\leq & (-2+ 2\beta^2)\int_{\mathbb R \times \mathbb R}  |x-y|^2 \pi(dx, dy),\\
\tilde L^{(2)}V(x- y)
= & 2\langle -2x+ 2y, x-y\rangle + |x-y|^2\\
= & -3|x-y|^2.
\end{align*}
Thus, $\tilde V(x)=\hat V(x)=x^2$ for $x\in \mathbb R$ and $\gamma=\theta=2-2\alpha^2$, i.e. (H5) and (H6) hold. Therefore, by Theorem \ref{thinvariantmeasureofDDSDE}, Corollary \ref{thinvariantmeasureunderTV}, Theorem \ref{thinvariantmeasure} and Corollary \ref{thinvariantmeasurestichingunderTV} we get the desired results.
\end{proof}

\begin{example}
Assume that for each $x\in \mathbb R$, $\mu \in \mathcal P (\mathbb R)$ and $i\in \{1,2\}$, \begin{align*}
b(x, \mu, 1)&= -x^3 -x,~ b(x, \mu, 2)= -\frac{1}{2} x,\\
\sigma(x, \mu, 1)&= \int_{\mathbb R} x \mu(dx),~ \sigma(x, \mu, 2)= x+2\int_{\mathbb R} x \mu(dx).
\end{align*}
$\alpha(t)$ is a two-state random jump process with $x$-dependent generator
$$
\begin{gathered}
\begin{pmatrix}
-\frac{1}{3}-\frac{1}{4}\cos x &  \frac{1}{3}+\frac{1}{4}\cos x\\
\frac{7}{3}+\frac{1}{2}\sin x &  -\frac{7}{3}-\frac{1}{2}\sin x\\
\end{pmatrix}
\end{gathered}.
$$
Then the following results hold: (i) there exists a unique solution $(X_t,\alpha_t)$ and when $ E|X_0| <\infty$ we have
$$
E|X_t| \leq  e^{-\frac{5}{12}t}E|X_0|, \quad \hbox{for } t\ge 0.
$$
(ii) When there is no switching, there exists a unique invariant measure to which the solutions' distributions are exponentially convergent under $W_1$ and $\|\cdot\|_{{\rm Var}, |\cdot|}$. When the generator of switching is state-independent, we obtain a unique invariant measure to which the solutions' distributions are exponentially convergent under $W_d$ and $\|\cdot\|_{{\rm Var}, d}$ where $d((x,i)(y,j))= \sqrt {\mathcal\mathbb 1_{i\neq j}+ |x-y|}$ for $x,y\in \mathbb R, ~ i,j\in \{1,2\}$.

\end{example}

\begin{proof}
(i) The coefficients $b$ and $\sigma$ clearly satisfy (H1), (H3) and (H4). Consider the Lyapunov function
$$
V(x,1 )=|x|, V(x,2 )=2|x|
$$
for $x\in \mathbb R$.
Then we have
\begin{align*}
 LV( x ,1 )
& =  {\rm sign} x \cdot( -x^3 -x ) + ( \frac{1}{3}+\frac{1}{4}\cos x )( 2- 1 )|x|\\
& \leq -|x| + \frac{7}{12} |x| =  -\frac{5}{12} |x|=  -\frac{5}{12}V( x ,1 ),\\
LV(x,2 )
& =  2 {\rm sign} x \times (-\frac{1}{2} x) + ( \frac{7}{3}+\frac{1}{2}\sin x )(1- 2 )|x|\\
& \leq - |x|-\frac{11}{6} |x|=-\frac{17}{12}V(x,2 ).
\end{align*}
Thus, in this example, $\varphi(x)= |x|$ for $x\in \mathbb R$, $\lambda_1= -\frac{5}{12},~\lambda_2= 0$, i.e. (H2) holds. Therefore, by Theorem \ref{thlL1}, there exists a unique solution $(X_t,\alpha_t)$ and we have
$$
E|X_t| \leq  e^{-\frac{5}{12}t}E|X_0|, \quad \hbox{for } t\ge 0.
$$

(ii) Let $\tilde V(x)= \hat V(x)= |x|$ for $x\in \mathbb R$. In the same way, we obtain
\begin{align*}
\tilde L^{(1)}\tilde V( x-y )
&= {\rm sign} (x-y) \cdot( -x^3 -x + y^3 +y) \\
& \leq  - |x-y|=  -\tilde V( x -y),\\
\tilde L^{(2)}\tilde V( x-y )&= {\rm sign} \left(x-y\right)(-\frac{1}{2} x +\frac{1}{2} y) \\
&=-\frac{1}{2}|x-y|=-\frac{1}{2}\tilde V(x-y),
\end{align*}
where $\pi\in \mathcal C(\mu, \nu)$.
Thus, $\gamma= 1, \theta= \frac{1}{2}$, i.e. (H5) and (H6) hold. Therefore, by Theorem \ref{thinvariantmeasureofDDSDE}, Corollary \ref{thinvariantmeasureunderTV}, Theorem \ref{thinvariantmeasure} and Corollary \ref{thinvariantmeasurestichingunderTV}, we get the desired results.
\end{proof}

\section*{Acknowledgements}

This work is partially supported by NSFC Grants 11871132, 11952102, Dalian High-level Talent Innovation Project (Grant 2020RD09), and Xinghai Jieqing fund from Dalian University of Technology.


\begin{thebibliography}{99}

\bibitem{BR}
G. Barone-Adesi and R. Whaley, Efficient analytic approximation of American option values, \it J. Finance \bf 42 \rm (1987),  301--320.

\bibitem{Barbu-Rockner}
V. Barbu and M. R\"ockner, From nonlinear Fokker-Planck equations to solutions of distribution dependent SDE, \it Ann. Probab. \bf48 \rm(2020), 1902--1920.

\bibitem{BLPR}
R. Buckdahn, J. Li, S. Peng and C. Rainer, Mean-field stochastic differential equations and associated PDEs, \it Ann. Probab. \bf 45 \rm (2017), 824--878.

\bibitem{Bogachev-Rockner-Shaposhnikov}
V. I. Bogachev, M. R\"ockner and S. V. Shaposhnikov,
Convergence in variation of solutions of nonlinear Fokker-Planck-Kolmogorov equations to stationary measures, \it J. Funct. Anal. \bf 276 \rm (2019), 3681--3713.

\bibitem{Butkovsky}
O. A. Butkovsky, On ergodic properties of nonlinear Markov chains and stochastic McKean-Vlasov equations, \it Theory Probab. Appl. \bf 58 \rm(2014), 661--674.

\bibitem{CM}
B. Cloez and M. Hairer, Exponential ergodicity for Markov processes with random switching, \it Bernoulli \bf 21  \rm(2015), 505--536.

\bibitem{HM}
M. Hairer and J. C. Mattingly, Yet another look at Harris' ergodic theorem for Markov chains, Seminar on Stochastic Analysis, Random Fields and Applications VI, 109--117,
Progr. Probab. 63, Birkh\"auser/Springer Basel AG, Basel, 2011.

\bibitem{HMS}
M. Hairer, J. C. Mattingly and M. Scheutzow, Asymptotic coupling and a general form of Harris' theorem with applications to stochastic delay equations, \it Probab. Theory Relat. Fields \bf 149 \rm (2011), 223--259.


\bibitem{HMC1}
M. Huang, R. P. Malham\'e and  P. E. Caines, Large population stochastic dynamic games: closed-loop
McKean-Vlasov systems and the Nash certainty equivalence principle, \it  Commun. Inf. Syst. \bf 6 \rm(2006), 221--251.

\bibitem{HMC2}
M. Huang, R. P. Malham\'e and  P. E. Caines, Large-population cost-coupled LQG problems with nonuniform agents: individual-mass behavior and decentralized $\epsilon$-Nash equilibria, \it  IEEE Trans. Automat. Control
\bf52 \rm(2007), 1560--1571.

\bibitem{Hu}
S. S. Hu, Long-time behavior for distribution dependent SDEs with local Lipschitz coefficients, \it arXiv preprint
\rm (2021), arXiv:2103.13101.

\bibitem{Kac}
M. Kac, Foundations of kinetic theory. In: Proceedings of the Third Berkeley Symposium on Mathematical Statistics and Probability,  III, 1954--1955, 171--197, Berkeley, Los Angeles: University of California Press, \rm(1956).

\bibitem{Larsy-Lions1}
J. M. Lasry and P. L. Lions, Jeux \`a champ moyen. I. Le cas stationnaire. (French) [Mean field games. I.
The stationary case], \it C. R. Math. Acad. Sci. Paris \bf343 \rm(2006), 619--625.

\bibitem{Larsy-Lions2}
J. M. Lasry and P. L. Lions,  Jeux \`a champ moyen. II. Horizon fini et controle optimal. (French) [Mean
field games. II. Finite horizon and optimal control], \it C. R. Math. Acad. Sci. Paris \bf343 \rm(2006), 679--684.

\bibitem{Larsy-Lions3}
J. M. Lasry and P. L. Lions, Mean field games, \it  Jpn. J. Math. \bf2 \rm(2007), 229--260.

\bibitem{Mckean}
H. P. McKean, Propagation of chaos for a class of nonlinear parabolic equations. In: Lecture Series in
Differential Equations \bf 7 \rm(1967),  41--57.

\bibitem{MV}
Y. Mishura and A. Veretennikov, Existence and uniqueness theorems for solutions of McKean-Vlasov stochastic equations,
{\it Theor. Probability and Math. Statist.} \bf103 \rm(2020), 59--101.




\bibitem{NYH}
S. L. Nguyen, G. Yin and T. A. Hoang, On law of large numbers for systems with mean-field interactions and Markovian switching, \it Stochastic Process. Appl. {\bf 130} \rm(2020), 262--296.

\bibitem{NYN}
S. L. Nguyen, G. Yin and  D. T. Nguyen, A general stochastic maximum principle for mean-field controls with regime switching, \it Appl. Math. Optim. \bf 84 \rm(2021), 3255--3294.

\bibitem{RTW}
P. Ren, H. Tang  and F. Y. Wang,  Distribution-path dependent nonlinear SPDEs with application to  stochastic transport type equationd, {\it arXiv preprint} \rm(2020), arXiv:2007.09188.

\bibitem{Shao1}
J. Shao, Ergodicity of regime-switching diffusions in Wasserstein distances, \it Stochastic Process. Appl. \bf125 \rm(2015),  739--758.

\bibitem{Shao2}
J. Shao, Strong solutions and strong Feller properties for regime-switching diffusion processes
in an infinite state space, \it SIAM J. Control Optim. \bf 53  \rm(2015),  2462--2479.


\bibitem{Song}
Y. Song, Gradient estimates and exponential ergodicity for mean-field SDEs with jumps, \it J. Theoret. Probab. \bf33 \rm(2020), 201--238 .

\bibitem{Sznitman1}
A. S. Sznitman, Nonlinear reflecting diffusion process, and the propagation of chaos
and fluctuations associated, \it J. Funct. Anal. \bf56 \rm(1984), 311--336 .

\bibitem{Sznitman2}
A. S. Sznitman, Topics in propagation of chaos. \'{E}cole d\'{E}t\'{e} de Probabilit\'{e}s de Saint-Flour XIX--1989, 165--251, Lecture Notes in Math. 1464, Springer, Berlin, \rm(1991).

\bibitem{Veretennikov}
A. Y. Veretennikov, On ergodic measures for McKean-Vlasov stochastic equations. In: Niederreiter H., Talay D. (eds) Monte Carlo and Quasi-Monte Carlo Methods 2004. Springer, Berlin, Heidelberg.\rm(2006).

\bibitem{Villani}
C. Villani, { \it Optimal Transport. Old and new}. Grundlehren der mathematischen Wissenschaften [Fundamental Principles of Mathematical Sciences], 338. Springer-Verlag, Berlin, \rm(2009), xxii+973 pp.

\bibitem{wang_2018}
F. Y. Wang, Distribution dependent SDEs for Landau type equations, {\it Stochastic Process. Appl.} {\bf 128} \rm(2018), 595--621.

\bibitem{Wang}
F. Y. Wang, Exponential ergodicity for non-dissipative McKean-Vlasov SDEs, {\it arXiv preprint} \rm(2021), arXiv:2101.12562 .

\bibitem{YKI}
G. Yin, V. Krishnamurthy and C. Ion, Regime switching stochastic approximation algorithms with application to adaptive discrete stochastic optimization, \it SIAM J. Optim. \bf 14 \rm(2004), 1187--1215.



\bibitem{YZ}
G. Yin and C. Zhu, {\it Hybrid Switching Diffusions: Properties and Applications}, Stochastic Modelling and Applied Probability,  63. Springer, New York, \rm(2010), xviii+395 pp.

\bibitem{ZSX}
X. Zhang, Z. Sun and  J. Xiong, A general stochastic maximum principle for a Markov regime switching jump-diffusion model of mean-field type, \it SIAM J. Control Optim. \bf 56 \rm(2018), 2563--2592.

\bibitem{ZY_2009}
C. Zhu  and  G. Yin, On strong Feller, recurrence, and weak stabilization of regime-switching diffusions,
 \it SIAM J. Comtrol Optim.  {\bf 48}  \rm (2009),  2003--2031.

\end{thebibliography}
\end{document}